\documentclass[a4paper,11pt,reqno,oneside]{amsart}
\usepackage{amsmath,amssymb,amsthm,mathtools}
\usepackage{mathrsfs,enumitem,xcolor,booktabs}
\usepackage[scale=0.75,hcentering]{geometry}
\usepackage{microtype,aliascnt,hyperref}
\usepackage[capitalize,nameinlink,noabbrev]{cleveref}
\makeatletter
\@ifundefined{subjclassname@2020}{\@namedef{subjclassname@2020}{%
\textup{2020} Mathematics Subject Classification}}{}
\makeatother
\definecolor{cobalt}{rgb}{0.0,0.28,0.67}
\hypersetup{bookmarksopen=true,
pdftitle={The Comon Rank Gap of a Third-Order Symmetric Tensor Can Exceed One},
pdfauthor={Jianze Li},
pdfsubject={The Comon Rank Gap of a Third-Order Symmetric Tensor Can Exceed One},
pdfkeywords={symmetric tensor, CP rank, Waring rank, Comon rank gap, direct sum},
colorlinks=true,linkcolor=cobalt,citecolor=cobalt,urlcolor=cobalt}
\theoremstyle{plain}
\newtheorem{theorem}{Theorem}[section]
\newaliascnt{corollary}{theorem}
\newtheorem{corollary}[corollary]{Corollary}
\aliascntresetthe{corollary}
\newaliascnt{lemma}{theorem}
\newtheorem{lemma}[lemma]{Lemma}
\newtheorem*{lemma*}{Lemma}
\aliascntresetthe{lemma}
\newaliascnt{proposition}{theorem}
\newtheorem{proposition}[proposition]{Proposition}
\aliascntresetthe{proposition}
\theoremstyle{definition}
\newaliascnt{definition}{theorem}
\newtheorem{definition}[definition]{Definition}
\aliascntresetthe{definition}
\newaliascnt{problem}{theorem}
\newtheorem{problem}[problem]{Open Problem}
\aliascntresetthe{problem}
\theoremstyle{remark}
\newaliascnt{remark}{theorem}
\newtheorem{remark}[remark]{Remark}
\newtheorem*{remark*}{Remark}
\aliascntresetthe{remark}
\newaliascnt{example}{theorem}

\aliascntresetthe{example}
\crefname{theorem}{Theorem}{Theorems}
\Crefname{theorem}{Theorem}{Theorems}
\crefname{corollary}{Corollary}{Corollaries}
\Crefname{corollary}{Corollary}{Corollaries}
\crefname{lemma}{Lemma}{Lemmas}
\Crefname{lemma}{Lemma}{Lemmas}
\crefname{proposition}{Proposition}{Propositions}
\Crefname{proposition}{Proposition}{Propositions}
\crefname{definition}{Definition}{Definitions}
\Crefname{definition}{Definition}{Definitions}
\crefname{problem}{Open Problem}{Open Problems}
\Crefname{problem}{Open Problem}{Open Problems}
\crefname{remark}{Remark}{Remarks}
\Crefname{remark}{Remark}{Remarks}
\crefname{example}{Example}{Examples}
\Crefname{example}{Example}{Examples}
\newcommand{\tens}[1]{\boldsymbol{#1}}

\newcommand{\matr}[1]{\mathbf{#1}}
\newcommand{\vect}[1]{\boldsymbol{#1}}
\newcommand{\set}[1]{\mathcal{#1}}
\newcommand{\eqdef}{\mathrel{\triangleq}}
\newcommand{\RR}{\mathbb{R}}
\newcommand{\CC}{\mathbb{C}}
\newcommand{\QQ}{\mathbb{Q}}
\newcommand{\FF}{\mathbb{F}}
\newcommand{\Sym}{\mathbf{Sym}}
\newcommand{\sym}{\operatorname{sym}}
\newcommand{\Span}{\operatorname{span}}
\newcommand{\rank}{\operatorname{rank}}
\newcommand{\Ess}{\operatorname{Ess}}
\newcommand{\mrank}{\operatorname{mrank}}
\newcommand{\diag}{\operatorname{diag}}
\newcommand{\rkone}{\mathcal{R}_1}
\newcommand{\cE}{\mathcal{E}}

\newcommand{\cB}{\mathcal{B}}

\newcommand{\im}{\operatorname{im}}
\newcommand{\supp}{\operatorname{supp}}
\newcommand{\id}{\operatorname{id}}
\newcommand{\DeltaT}{\Delta}
\newcommand{\J}{\mathcal{J}}
\newcommand{\cA}{\mathcal{A}}
\newcommand{\cK}{\mathcal{K}}
\newcommand{\cV}{\mathcal{V}}
\newcommand{\cW}{\mathcal{W}}
\newcommand{\cZ}{\mathcal{Z}}
\newcommand{\va}{\vect{a}}
\newcommand{\vb}{\vect{b}}
\newcommand{\vc}{\vect{c}}
\newcommand{\vd}{\vect{d}}
\newcommand{\vh}{\vect{h}}
\newcommand{\vu}{\vect{u}}
\newcommand{\vv}{\vect{v}}
\newcommand{\vw}{\vect{w}}
\newcommand{\vx}{\vect{x}}
\newcommand{\vy}{\vect{y}}
\newcommand{\vz}{\vect{z}}

\numberwithin{equation}{section}
\allowdisplaybreaks[2]
\setlist[enumerate]{label=\textup{(\roman*)},leftmargin=*}
\title[The Comon Rank Gap of a Third-Order Symmetric Tensor Can Exceed One]
{The Comon Rank Gap of a Third-Order Symmetric Tensor Can Exceed One}
\author{Jianze Li}
\address{School of Science, Sun Yat-sen University, Shenzhen, Guangdong, China}
\email{lijianze@mail.sysu.edu.cn}
\date{September 28, 2026}
\subjclass[2020]{15A69, 14N07, 15A03}
\keywords{symmetric tensor, CP rank, symmetric rank, Comon rank gap,
Waring rank, direct sum, substitution}
\begin{document}
\begin{abstract}
For a third-order symmetric tensor, the Comon rank gap is the difference
between its CP rank and its symmetric rank. We study direct sums of
Lovitz's rational $27$-dimensional tensor, whose two ranks are $55$
and $56$. Two copies have ranks $110$ and $112$, and three copies
have ranks $165$ and $168$, over both the real and the complex fields.
Thus direct sums yield concise cubic tensors on spaces of dimensions
$54$ and $81$ with Comon rank gaps two and three. 
\end{abstract}
\maketitle

\section{Introduction}\label{sec:introduction}
For $\FF\in\{\RR,\CC\}$ and $\tens{T}\in\Sym^3(\set{U})$, let
$r_{\FF}(\tens{T})$ and $s_{\FF}(\tens{T})$ denote its CP rank \cite{kolda2009tensor,landsberg2012tensors,comon2014tensors}
and symmetric rank \cite{comon2008symmetric} over $\FF$, respectively. The \emph{Comon rank gap} is defined as
\[
\DeltaT_{\FF}(\tens{T})\eqdef s_{\FF}(\tens{T})-r_{\FF}(\tens{T})\ge0.
\]
A central problem in this area is \emph{Comon's conjecture}
\cite{comon2004canonical,comon2008symmetric,comon2014tensors}, which
asks whether
$\DeltaT_{\FF}(\tens{T})=0$ 
holds for every real or complex symmetric tensor $\tens{T}$. 
The conjecture has nevertheless been established under several additional
hypotheses. Over $\CC$, the foundational work
\cite{comon2008symmetric} proves equality in basic low-symmetric-rank cases
and under suitable decomposition conditions. Zhang, Huang, and Qi establish
equality when the CP rank does not exceed the tensor order
\cite[Corollary~4.3]{zhang2016comon}, while Friedland's criterion covers
CP rank at most the flattening rank plus one
\cite[Theorem~1.1]{friedland2016remarks}. Further positive results include
tensors of sufficiently low symmetric rank \cite{ballico2011lowrank},
tensors on tangential varieties of Veronese varieties
\cite[Corollary~2]{ballico2013tensor}, and complex cubic surfaces
\cite[Theorem~1.3]{seigal2020cubic}; Seigal also proves equality for complex
cubic tensors of symmetric rank at most seven
\cite[Corollary~1.4]{seigal2020cubic}.
We refer to
\cite{casarotti2020comon} for a broader survey.

The search for counterexamples took a different direction when Shitov
introduced the symmetric-adjoin framework in
\cite{shitov2018counterexample}. Although the specific example in that paper
was subsequently shown to contain an error \cite{draisma2024erratum}, new
counterexamples over $\CC$ have recently been reported in
\cite{shitovhigher,Lovitz26:arxiv}.
Over $\RR$, counterexamples are known in orders four and six
\cite{shitov2020comon,wang2023lower}, while order-three counterexamples have
also been reported in \cite{shitovhigher,Lovitz26:arxiv}. 
In particular, in low dimension $27$, Lovitz
\cite[Theorem~1.1]{Lovitz26:arxiv} constructed a tensor
$\tens{T}_0\in\Sym^3(\QQ^{27})$ whose scalar extensions satisfy
\begin{equation}\label{eq:T0_Lovitz}
r_{\FF}(\tens{T}_0)=55,
\ \ 
s_{\FF}(\tens{T}_0)=56,
\ \ 
\FF\in\{\RR,\CC\}.
\end{equation}
These counterexamples bring into focus a natural question posed over $\CC$ in \cite[Section~6, Problem~23]{shitov2018counterexample} and reiterated in \cite{Lovitz26:arxiv}:
\begin{equation}\label{eq-question-comon}
\textit{Is it possible that }\DeltaT_{\FF}(\tens{T})>1\textit{?}
\end{equation}

In \cite[Section~6]{shitov2018counterexample}, a direct-sum strategy was proposed as a possible approach to resolving the question above. In this paper, we pursue this strategy and answer the question affirmatively by studying direct sums of two and three independent copies of the tensor $\tens{T}_0$ introduced in \cite{Lovitz26:arxiv}.
Establishing this affirmative answer requires lower bounds tailored to these
direct sums. Indeed, CP rank is not additive in general
\cite{shitov2019directsum}, although positive additivity results are known for
small three-way tensors
\cite{buczynski2020additivity,rupniewski2024additivity}. Symmetric-rank
additivity poses a separate challenge, with results currently available for
specific classes \cite{carlini2015symmetricstrassen,carlini2018ecomputability}.
Consequently, the evident upper bounds
\[
r_{\FF}(\tens{T}_0^{\oplus k})\le55k,
\ \  s_{\FF}(\tens{T}_0^{\oplus k})\le56k
\]
do not by themselves determine the gap. The following theorem supplies the
matching lower bounds for two and three copies.

\begin{theorem}[Cubic tensors with gaps two and three]\label{thm:main}
Let $\tens{T}_0$ be the rational tensor defined in
\cite{Lovitz26:arxiv}, and set
$\tens{V}_k\eqdef\tens{T}_0^{\oplus k}$, where the direct sum is
taken over $k$ independent copies of the ambient space. For
$k\in\{2,3\}$ and $\FF\in\{\RR,\CC\}$, we have
\begin{equation}\label{eq:main}
r_{\FF}(\tens{V}_k)=55k,\ \ 
s_{\FF}(\tens{V}_k)=56k,\ \ 
\DeltaT_{\FF}(\tens{V}_k)=k.
\end{equation}
The tensor $\tens{V}_k$ has rational entries and is concise on a
$27k$-dimensional space.
\end{theorem}

Since $(r_{\FF}(\tens{T}_0),s_{\FF}(\tens{T}_0))=(55,56)$,
\cref{thm:main} shows that both ranks are additive for these specified double
and triple copies over either field. The double-copy proof is developed
directly in terms of $\tens{V}_2$. For the triple copy,
\cref{sec:triple} develops the fixed-source reductions and proves the
$167$-term lower bound, while \cref{sec:triple-167} classifies and excludes
the remaining $167$-term case.

\section{The single-copy construction and local rigidity}
\label{sec:construction}
\subsection{Notation and elementary rank facts}
All spaces are finite-dimensional over $\FF\in\{\RR,\CC\}$.
Tensors are denoted by bold symbols such as $\tens{P}$, matrices and
linear maps by $\matr{M}$ and $\matr{L}$, vectors by $\vect{a}$,
and spaces and sets by calligraphic symbols.
We suppress the field subscript on ranks when the field is fixed.
Every scalar has a cube root in either field, so a symmetric
decomposition may be written as $\sum_i\vh_i^{\otimes3}$. 
For $\tens R\in\cV^{\otimes3}$, set
$\sym(\tens R)\eqdef\frac{1}{3!}\sum_{\sigma\in S_3}\sigma\cdot\tens R$.
We write $\vu^k$ for $\vu^{\otimes k}$. For products of vectors,
juxtaposition denotes symmetrization; more generally, for
$\matr M\in\Sym^2(\cV)$, we set
$\vu\matr M\eqdef\sym(\vu\otimes\matr M)$ and extend this notation
linearly and setwise; for a subspace $\set W\subseteq\cV$, examples are
$\set W\matr M$ and $\vu\set W$. Thus
$\vu\vv\vw\eqdef\sym(\vu\otimes\vv\otimes\vw)$,
whereas all three-factor residual expressions below use $\otimes$.
For $\tens{P}\in\set{U}_A\otimes\set{U}_B\otimes\set{U}_C$, set
\[
\begin{aligned}
\Ess_A(\tens{P})\eqdef\im\bigl((\set{U}_B\otimes\set{U}_C)^*
\xrightarrow{\ \lrcorner\,\tens{P}\ }\set{U}_A\bigr),\ \
\mrank(\tens{P})\eqdef(\dim\Ess_A\tens{P},\dim\Ess_B\tens{P},
\dim\Ess_C\tens{P}).
\end{aligned}
\]
The other mode spaces are defined cyclically. The tuple
$\mrank(\tens P)$ is the \emph{multilinear rank}; when considered up
to mode permutation, it is also called the \emph{mode type}.
A symmetric tensor on $\set{U}$ is \emph{concise} when
$\Ess_A(\tens{P})=\set{U}$.
We write $\vx\parallel\vy$ only for proportional nonzero vectors;
$\rkone(\set{E})\eqdef\{\matr{Q}\in\set{E}:\rank\matr{Q}=1\}$.

Rank does not change on embedding a tensor into larger factor spaces;
see \cite[Lemma~2.8]{buczynski2020additivity} for CP rank and
\cite[Remark~2.3]{carlini2018ecomputability} for symmetric rank.
The projection arguments establishing these statements are
field-independent and apply over both $\RR$ and $\CC$.
The same CP-rank lemma puts all factors of a minimal CP decomposition in the
mode spaces. Its two-factor products are linearly independent
(otherwise one term is redundant), so the factor spans equal those spaces. In particular, a tensor of CP rank three and
multilinear rank $(3,3,3)$ has three independent factors in each mode in
every three-term CP expression.
\begin{lemma}[Rank and elimination with independent factors]
\label{lem:factor-calculus}
Let $\matr B\eqdef[\vb_1\ \cdots\ \vb_n]$ and
$\matr C\eqdef[\vc_1\ \cdots\ \vc_n]$ have ranks $b,c$.
For $\set I\subseteq[n]$ and $t_i\ne0$ $(i\in\set I)$, we have
\begin{equation}\label{eq:factor-support-rank}
\rank\!\left(\sum_{i\in\set I}t_i\vb_i\otimes\vc_i\right)
\ge |\set I|-(n-b)-(n-c).
\end{equation}
If both factor lists are independent, this rank is exactly $|\set I|$,
and its column and row spaces are the spans of their active factors.
If the $\va$-factor list, the $\vb$-factor list, and the $\vc$-factor
list are all independent, put
\[
\tens P\eqdef\sum_{i=1}^n\va_i\otimes\vb_i\otimes\vc_i,
\ \  \matr N\eqdef\sum_{i=1}^nt_i\vb_i\otimes\vc_i\ne0.
\]
After relabelling so that $t_1\ne0$, we have
\begin{equation}\label{eq:factor-elimination}
\tens P-(\va_1/t_1)\otimes\matr N
=\sum_{i=2}^n(\va_i-t_i\va_1/t_1)\otimes\vb_i\otimes\vc_i
\end{equation}
has CP rank $n-1$ and all three mode dimensions equal to $n-1$.
\end{lemma}
\begin{proof}
Deleting $n-|\set I|$ columns gives
$\rank\matr B_{\set I}\ge b-(n-|\set I|)$ and the analogous bound for
$\matr C_{\set I}$. Sylvester's inequality
\cite[Theorem~3.6]{wang2023lower}, applied to
$\matr B_{\set I}\diag(t_i)\matr C_{\set I}^\top$ proves
\eqref{eq:factor-support-rank}. With independent lists, left
inverses give the rank equality and the stated supports.
In \eqref{eq:factor-elimination}, all three remaining lists are
independent; for the first, compare coefficients on
$\va_2,\ldots,\va_n$. Their flattening ranks give the conclusion.
\end{proof}

\subsection{Symmetric adjoining and its affine space}
For $\cK\subseteq\Sym^2(\cV)$, let
\begin{equation}\label{eq:J}
\J_{\cV}(\cK)\eqdef
\cV\otimes\cK+\iota_2(\cV\otimes\cK)+\cK\otimes\cV,
\end{equation}
where $\iota_2$ interchanges the first two positions.
Given $\tens{F}\in\Sym^3(\cV)$ and an independent family of symmetric
rank-one matrices $\matr{M}_1,\ldots,\matr{M}_m$ spanning $\cK$,
take a new space $\cZ$ with basis $\vect{e}_1,\ldots,\vect{e}_m$.
Its symmetric adjoin is
\begin{equation}\label{eq:adjoin}
\tens{T}\eqdef\tens{F}+3\sum_{\nu=1}^m\vect{e}_\nu\matr{M}_\nu
\in\Sym^3(\cV\oplus\cZ),\ \ 
\cA\eqdef\tens{F}+\J_{\cV}(\cK).
\end{equation}
The affine space $\cA$ consists of tensors that are not necessarily
symmetric. The substitution method goes back to
\cite{hopcroft1971minimizing}; see also the layer-reduction statements
\cite[Lemma~B.1 and Corollary~B.2]{alexeev2011tensor} and
\cite{wang2023lower}.
For \eqref{eq:adjoin}, we use
\begin{equation}\label{eq:adjoin-rank}
r(\tens{T})=3m+\min_{\tens{P}\in\cA}r(\tens{P})
\end{equation}
from \cite[Lemma~3.4(iii)]{Lovitz26:arxiv}, together with its refinement
\cite[Lemma~3.4(i)--(ii)]{Lovitz26:arxiv}.
\subsection{The rational nine-dimensional base}
Put $\vect{\alpha}\eqdef(1,1,-2)$ and
\begin{equation}\label{eq:base-data}
\begin{split}
\set{G}&\eqdef\{(\vu,\vv)\in\FF^3\oplus\FF^3:
2\textstyle\sum u_i+\sum v_i=0\},\\
\va&\eqdef(\vect{\alpha},-\vect{\alpha}),\ \ 
\vc\eqdef(\vect{\alpha},\vect{\alpha}),\ \ 
\vw_{pq}\eqdef(\vect{e}_p,-2\vect{e}_q)\ (1\le p,q\le3).
\end{split}
\end{equation}
Choose a rational complement $\set{G}'$ to $\FF\vc$ in $\set{G}$,
and glue two copies of $\set{G}$ along that line:
$\cV_0\eqdef\FF\vc\oplus(\set{G}')^{(1)}\oplus(\set{G}')^{(2)}$.
The superscript here denotes the two internal sectors of a
\emph{single} copy, and we put
$\set G^{(i)}\eqdef\FF\vc\oplus(\set G')^{(i)}$.
Write $\va^{(i)}$ and $\vw_{i,p,q}$ for the corresponding vectors,
and set
\begin{equation}\label{eq:single-FK}
\tens{F}_0\eqdef(\va^{(1)})^3+(\va^{(2)})^3+\vc^3,\ \ 
\matr{M}_{i,p,q}\eqdef\alpha_p\alpha_q\vw_{i,p,q}^{\,2},\ \ 
\cK_0\eqdef\Span\{\matr{M}_{i,p,q}\}.
\end{equation}
Let $\cZ_0$ have basis
$\{\vect e_{i,p,q}:i\in\{1,2\},\ 1\le p,q\le3\}$, and, when
convenient, enumerate $\vw_{i,p,q},\matr M_{i,p,q},\vect e_{i,p,q}$
as $\vw_\nu,\matr M_\nu,\vect e_\nu$ for $1\le\nu\le18$.
The tensor $\tens{T}_0$ in \cref{thm:main} is the adjoin
\eqref{eq:adjoin} of these data, with $m=18$. 
The following facts are \cite[Lemma~4.1, Theorem~4.2,
equation~(4.2), and Theorem~1.1]{Lovitz26:arxiv}, respectively.
The real affine-space exclusion follows from the complex statement
by complexification.
\begin{proposition}[Imported single-copy facts]\label{prop:single}
The eighteen matrices in \eqref{eq:single-FK} are independent.
The affine space $\cA_0\eqdef\tens{F}_0+\J_{\cV_0}(\cK_0)$ contains
neither zero nor a tensor of symmetric rank one, and it contains
\begin{equation}\label{eq:single-representatives}
\tens{P}_0\eqdef(\va^{(1)}+\va^{(2)}+\vc)\otimes\vc\otimes\vc.
\end{equation}
Moreover, we have $r(\tens{T}_0)=55$ and $s(\tens{T}_0)=56$ over both fields.
\end{proposition}

Since $0\notin\cA_0$, we have
$\cA_0\cap\J_{\cV_0}(\cK_0)=\varnothing$.
The additional linear map used below can be described without
choosing coordinates for the complement $\set{G}'$.
Let $\matr{L}:\cV_0\to\FF^3$ restrict to the second coordinate
projection on the first internal sector and to the first coordinate
projection on the second sector. These restrictions agree on $\vc$.
Writing its rows as $\varphi_1,\varphi_2,\varphi_3$, we have
\begin{equation}\label{eq:selector}
\matr{L}\vw_{1,p,q}=-2\vect{e}_q,\ 
\matr{L}\vw_{2,p,q}=\vect{e}_p,\ 
\matr{L}\va^{(1)}=-\vect{\alpha},\ 
\matr{L}\va^{(2)}=\matr{L}\vc=\vect{\alpha}.
\end{equation}
A vector is called \emph{selector-dense} when all three coordinates
of its image under $\matr L$ are nonzero; the same term is used
blockwise in a direct sum.
Thus $\matr{L}\matr{Q}\matr{L}^\top$ is diagonal for
$\matr{Q}\in\cK_0$, and, for pairwise distinct $r,s,t$, we have
\begin{equation}\label{eq:six}
(\varphi_r\otimes\varphi_s)(\cK_0)=0,\ \ 
(\varphi_r\otimes\varphi_s\otimes\varphi_t)(\tens{P})=-2
\ (\tens{P}\in\cA_0).
\end{equation}
Indeed, the value on $\tens{F}_0$ is
$(-\alpha_r)(-\alpha_s)(-\alpha_t)+2\alpha_r\alpha_s\alpha_t=-2$,
and every affine correction has two distinct row functionals on
its $\cK_0$ positions.

\begin{lemma}[Faithful one-sided contraction]\label{lem:faithful}
The map $\matr{Q}\mapsto\matr{Q}\matr{L}^\top$ is injective
on $\cK_0$.
\end{lemma}
\begin{proof}
Write $\matr{Q}=\sum_{i,p,q}t_{i,p,q}\matr{M}_{i,p,q}$.
For fixed $r$, only
\[
\set{U}_{1r}\eqdef\Span\{\vw_{1,p,r}:1\le p\le3\},\ \ 
\set{U}_{2r}\eqdef\Span\{\vw_{2,r,q}:1\le q\le3\}
\]
contribute to $\matr{Q}\varphi_r$. Each has dimension three, and we have
$\set{U}_{1r}\cap\set{U}_{2r}\subseteq\FF\vc$ as well as
$\matr{L}(\set{U}_{1r}+\set{U}_{2r})\subseteq\FF\vect{e}_r$, whereas
$\matr{L}\vc=\vect{\alpha}\notin\FF\vect{e}_r.$ 
Thus the intersection is zero. Since all $\alpha_p\ne0$, it follows that
\[
\matr{Q}\varphi_r=0
\ \Longrightarrow\
t_{1,p,r}=t_{2,r,q}=0\ \text{for all }p,q.
\]
Applying this for every $r$ gives $\matr{Q}=0$.
\end{proof}

\begin{lemma}[Local factor rigidity]\label{lem:local}
If $\tens{P}=\vu\otimes\vv\otimes\vw\in\cA_0$ has rank one, then
each of $\matr{L}\vu,\matr{L}\vv,\matr{L}\vw$ has three nonzero
coordinates, none of $\vu^2,\vv^2,\vw^2$ belongs to $\cK_0$, and
\begin{equation}\label{eq:zero-intersection}
\dim\Span\{\vu,\vv,\vw\}=2,\ \ 
\cK_0\cap\Sym^2(\Span\{\vu,\vv,\vw\})=0.
\end{equation}
More generally, if
$\tens P=\vu\otimes\matr C\in\cA_0$ in some mode ordering, then $\vu$ has three nonzero
$\matr L$-coordinates and $\vu^2\notin\cK_0$.
\end{lemma}
\begin{proof}
For $\tens P=\vu\otimes\matr C$, put
$c_{st}\eqdef(\varphi_s\otimes\varphi_t)(\matr C)$. Then, for each
$r$, \eqref{eq:six} gives
\[
\varphi_r(\vu)\,c_{st}=-2\ (\{r,s,t\}=\{1,2,3\})
\ \Longrightarrow\  \varphi_r(\vu)\ne0.
\]
Hence $(\matr{L}\vu)(\matr{L}\vu)^\top$ has nonzero
off-diagonal entries, whereas $\matr{L}\cK_0\matr{L}^\top$
is diagonal; thus $\vu^2\notin\cK_0$.
For a rank-one representative, put
$\cW\eqdef\Span\{\vu,\vv,\vw\}$. By
\cite[Lemma~3.7]{Lovitz26:arxiv}, we have
\[
\dim\cW=2,\ \ 
\cK_0\cap\Sym^2\cW\in\{0,\FF\vx^2\},
\]
and the second alternative requires one factor parallel to $\vx$.
The square exclusion eliminates that alternative.
\end{proof}

\begin{remark}[Relation to single-copy rigidity]\label{rem:local-literature}
The selector is the first map in
\cite[proof of Theorem~4.2]{Lovitz26:arxiv}.
It strengthens the factor-plane dichotomy of
\cite[Lemma~3.7]{Lovitz26:arxiv} by excluding every factor square
from $\cK_0$. Its one-sided injectivity in \cref{lem:faithful}
is a separate property, proved above from the explicit generators.
\end{remark}

\begin{definition}[Good and bad mode pairs]\label{def:good}
A mode pair of a rank-one representative in $\cA_0$, or in any
independent copy of it, is good if its two factors are independent
and bad if they are proportional.
\end{definition}
Every such representative has at least two good pairs, by
\cref{lem:local}. Consequently two representatives have a common
good pair. This elementary observation permits a permutation of the
original substitution maps; it does not require new maps. 
By \cite[proof of Theorem~3.2, Case~1]{Lovitz26:arxiv}, the
zero intersection in \cref{lem:local} gives the following input.
\begin{lemma}[Rank-one matrices in a local extension]
\label{lem:matrix-rigidity}
If $\vu\otimes\vv\otimes\vw\in\cA_0$ and $\vu,\vv$ are independent,
then
\[
\rkone\bigl(\cK_0+\FF(\vu\otimes\vv)\bigr)\setminus\cK_0
=\{\lambda\vu\otimes\vv:\lambda\in\FF\setminus\{0\}\}.
\]
\end{lemma}

\begin{lemma}[Binary reduction modulo a quadratic]\label{lem:binary}
If $\dim\cW=2$ and $\cK_0\cap\Sym^2(\cW)$ contains a rank-two
matrix, then $\cA_0\cap\Sym^3(\cW)=\varnothing$.
\end{lemma}
\begin{proof}
This is the binary reduction in \cite[Lemma~3.6(i) and proof of
Lemma~3.7]{Lovitz26:arxiv}: for the stated rank-two $\matr Q$,
each $\tens H\in\Sym^3\cW$ satisfies
$\tens H-\vh^3\in\cW\matr Q\subseteq\J_{\cV_0}(\cK_0)$
for some $\vh\in\cW$. Membership of $\tens H$ in $\cA_0$
would contradict the cube exclusion in \cref{prop:single}.
\end{proof}

\begin{remark}\label{rem:symmetrization}
The affine space is invariant under symmetrization. Thus the last
contradiction also applies to a representative, not necessarily symmetric, supported
on $\cW^{\otimes3}$. It does not apply if one of its factors lies
outside $\cW$.
This is the same symmetrization step used in
\cite[proof of Lemma~3.7]{Lovitz26:arxiv}, after the binary
reduction of \cite[Lemma~3.6(i)]{Lovitz26:arxiv}.
\end{remark}

\section{Direct sums and the fixed-source residual principle}
\label{sec:source}
\subsection{Block notation and the rank of direct sums}
Use subscripts $j=1,2$ for two independent copies of the entire
single-copy construction. These are different from its internal
sectors. Set $\cV\eqdef\cV_1\oplus\cV_2$ and
$\cZ\eqdef\cZ_1\oplus\cZ_2$, and set $\cK\eqdef\cK_1\oplus\cK_2$ and
$\cA_j\eqdef\tens F_j+\J_{\cV_j}(\cK_j)$ and
$\cA\eqdef\tens{F}_1+\tens{F}_2+\J_{\cV}(\cK)$.
When the ambient space is clear, we abbreviate
$\J_{\cV_j}(\cK_j)$ and $\J_{\cV}(\cK)$ by $\J(\cK_j)$ and
$\J(\cK)$, respectively. Let $\matr{p}_j$ denote block projection,
and write $\vx^{[j]}\eqdef\matr p_j\vx$ for the block-$j$ component
of a vector. For $\tens{P}\in\cV^{\otimes3}$ write
$\tens{P}_{ijk}$ for its component in
$\cV_i\otimes\cV_j\otimes\cV_k$.
We call $\cV_j$ a physical block and $\tens P_{jjj}$ the
\emph{pure-block component}, or simply the \emph{pure block}, at $j$;
its CP rank is the
\emph{pure-block CP rank} at $j$. A tensor or vector is \emph{pure}
when it is supported on one physical block. The pure-block components
of a tensor in $\cA$ belong to $\cA_j$.
Its $(2,1,1)$ component belongs to $\cV_2\otimes\cK_1$, with the
corresponding statement in the other five mixed positions.
In particular the affine space includes cross-block corrections.
For general $k$, write
$\cA^{[k]}\eqdef\sum_{j=1}^k\tens F_j+
\J_{\bigoplus_j\cV_j}(\bigoplus_j\cK_j)$.

\begin{proposition}[Affine minimum, CP rank, and concision of all copies]
\label{prop:ordinary}
For every integer $k\ge1$, we have
$\min_{\tens P\in\cA^{[k]}}r(\tens P)=k$. Moreover,
$\tens{V}_k$ is concise in dimension $27k$ and satisfies
$r(\tens{V}_k)=55k$ over either field.
\end{proposition}
\begin{proof}
Temporarily set $\cV\eqdef\bigoplus_{j=1}^k\cV_j$ and
$\cK\eqdef\bigoplus_{j=1}^k\cK_j$.
For $q=1,2,3$, let $\matr{L}_q:\cV\to\FF^k$ apply row
$\varphi_q$ of \eqref{eq:selector} separately on each copy:
\[
\matr{L}_q\Bigl(\sum_{j=1}^k\vv_j\Bigr)
\eqdef\sum_{j=1}^k\varphi_q(\vv_j)\vect{e}_j.
\]
Every correction in $\J_{\cV}(\cK)$ is annihilated by
$\matr{L}_1\otimes\matr{L}_2\otimes\matr{L}_3$, since its
two $\cK$ positions receive distinct selector rows. Hence
\begin{equation}\label{eq:all-copy-selector}
(\matr{L}_1\otimes\matr{L}_2\otimes\matr{L}_3)\tens{P}
=-2\sum_{j=1}^k\vect{e}_j^3
\ \ (\tens{P}\in\cA^{[k]}).
\end{equation}
The right-hand tensor has flattening rank $k$, so every such
$\tens{P}$ has CP rank at least $k$. The sum of the $k$
representatives in \eqref{eq:single-representatives} attains this
bound, proving the affine minimum. Equation~\eqref{eq:adjoin-rank},
with $m=18k$, gives
$r(\tens{V}_k)=54k+k=55k$. 
For concision, the annihilator of $\{\vw_{pq}\}$ in
$(\FF^3\oplus\FF^3)^*$ is
$\{(\xi,\eta):\xi_p-2\eta_q=0\ \forall p,q\}
=\FF(2,2,2,1,1,1).$ 
Hence $\Span\{\vw_{pq}\}=\set{G}$ and the eighteen adjoin
directions span $\cV_0$. For
$\xi+\zeta\in\cV_0^*\oplus\cZ_0^*$, we have
\[
(\xi+\zeta)\lrcorner\tens{T}_0=0
\ \Longrightarrow\
\xi(\vw_\nu)\vw_\nu=0\ (1\le\nu\le18)
\ \Longrightarrow\ \xi=0.
\]
Then $\sum_\nu\zeta(\vect{e}_\nu)\matr{M}_\nu=0$ gives
$\zeta=0$. Thus $\tens{T}_0$ is concise. The flattening images
of its independent copies form a direct sum, so
$\dim\Ess_A(\tens{V}_k)=27k$.
\end{proof}

We next strengthen the affine minimum to a structural statement.
\begin{lemma}[Rank-two block separation]\label{lem:rank-two}
If $\tens{P}\in\cA$ and $r(\tens{P})=2$, then
$\tens{P}=\tens{P}_1+\tens{P}_2$, where
$\tens{P}_j\in\cA_j$ has rank one and is supported entirely on
$\cV_j^{\otimes3}$. In every minimal two-term CP decomposition, the two
terms are these pure-block terms, up to order.
\end{lemma}
\begin{proof}
Write $\tens{P}=\sum_{i=1}^2\va_i\otimes\vb_i\otimes\vc_i$.
For example, we have
\[
\dim\Span\{\matr{p}_2\va_i:i=1,2\}=2
\ \Longrightarrow\
\matr{p}_1\vb_i\otimes\matr{p}_1\vc_i\in\cK_1\ (i=1,2)
\]
by the $(2,1,1)$ equation, contradicting
$\tens{P}_{111}\in\cA_1$. Permuting modes and blocks gives
\[
\dim\Span\{\matr{p}_j\va_i\}_i
=\dim\Span\{\matr{p}_j\vb_i\}_i
=\dim\Span\{\matr{p}_j\vc_i\}_i=1\ (j=1,2).
\]
Use pure-block bases $\va_j',\vb_j',\vc_j'$ in these six lines:
\[
\tens{P}=\sum_{j,k,\ell=1}^2
t_{jk\ell}\va_j'\otimes\vb_k'\otimes\vc_\ell',
\ \  t_{111}t_{222}\ne0.
\]
Every same-block pair lies outside $\cK_j$, since otherwise the
corresponding pure block lies in $\J(\cK_j)$. Each mixed equation
therefore gives $t_{jk\ell}=0$ unless $j=k=\ell$.
Their sum has multilinear rank $(2,2,2)$, so Kruskal's
rank-two uniqueness criterion
\cite[Corollary~2]{rhodes2010concise} gives the last assertion.
\end{proof}

\subsection{Compatible residuals and the original substitution maps}
Suppose $\tens{V}_2=\sum_{i=1}^N\vh_i^3$ is a minimal symmetric
decomposition with $N\le111$. By \cref{prop:ordinary}, we have $N\ge110$.
The decomposition-level substitution lemma
\cite[Lemma~3.4(i)--(ii)]{Lovitz26:arxiv} supplies maps
$\matr{\Pi}_A,\matr{\Pi}_B,\matr{\Pi}_C:\cV\oplus\cZ\to\cV$
fixing $\cV$, and pairwise disjoint index sets
$\set{I}_A,\set{I}_B,\set{I}_C$ of size $36$, such that
$\matr{\Pi}_m\vh_i=0$ for $i\in\set{I}_m$.
Consequently
\begin{equation}\label{eq:residual}
\tens{P}\eqdef(\matr{\Pi}_A\otimes\matr{\Pi}_B\otimes\matr{\Pi}_C)
\tens{V}_2\in\cA,\ \  2\le r(\tens{P})\le N-108\le3.
\end{equation}
All subsequent mode permutations permute these original maps and
their original index sets.

For any selected pair, written as $A,B$, put
\begin{equation}\label{eq:G-original}
\tens{G}\eqdef(\matr{\Pi}_A\otimes\matr{\Pi}_B\otimes\id)\tens{V}_2
=\sum_{i\in\set{S}}\matr{Q}_i\otimes\vh_i,\ \ 
\matr{Q}_i\eqdef(\matr{\Pi}_A\vh_i)\otimes(\matr{\Pi}_B\vh_i).
\end{equation}
Here $\set S\eqdef\{i:\matr Q_i\ne0\}$, so zero matrices have
been discarded and $|\set{S}|\le N-72\le39$.
Writing the square bases of both copies as
$\matr{M}_1,\ldots,\matr{M}_{36}$, direct expansion gives
\begin{equation}\label{eq:G-compensation}
\tens{G}=\tens{P}
+\sum_{\nu=1}^{36}\matr{M}_\nu\otimes
(\vect{e}_\nu-\matr{\Pi}_C\vect{e}_\nu).
\end{equation}
The $36$ compensation vectors have independent projections on
$\cZ$, whereas all residual coefficients belong to $\cV$.
We call a residual together with these maps and the expression
\eqref{eq:G-original} \emph{compatible with the original
$N$-term symmetric decomposition}. This is an additional condition, not merely
membership in $\cA$.

For the selected pair $A,B$, write
\[
\cE(\tens{G})\eqdef\im\bigl((\cV\oplus\cZ)^*
\xrightarrow{\ \lrcorner\,\tens{G}\ }\cV\otimes\cV\bigr),
\ \  \set L(\tens P)\eqdef\im\bigl(\cV^*
\xrightarrow{\ \lrcorner\,\tens P\ }\cV\otimes\cV\bigr).
\]
Here and below, the notation $\cE(\tens G)$ denotes the complete matrix
contraction image for the selected pair, while $\set L(\tens P)$
denotes the residual slice space. The notations $\cE_{XY}$ and
$\set L_{XY}(\tens P)$ abbreviate these spaces when the pair $XY$
is displayed. For a physical block $j$, the space
$(\matr p_j\otimes\matr p_j)\cE(\tens G)$ is the local contraction
image at $j$. We call
$\vh_i^3$ a decomposition term and $\matr Q_i$ its
matrix source for the selected pair, and
$(\matr p_j\otimes\matr p_j)\matr Q_i$ its local source matrix at
block $j$; below, ``source'' refers only to such a matrix source or
its index.

\begin{lemma}[The contraction-image criterion]\label{lem:image-criterion}
For any residual and compensation as in \eqref{eq:G-compensation}, we have
\begin{equation}\label{eq:image-exact}
\cE(\tens{G})=\cK+\set L(\tens{P}),
\end{equation}
where $\set L(\tens{P})$ is its residual slice space for the selected pair.
In particular, if $\tens{P}=\sum_{t=1}^q\matr{R}_t\otimes\vc_t$
with independent $\vc_t$ and independent classes
$\overline{\matr{R}}_t$ modulo $\cK$, then
\begin{equation}\label{eq:image-criterion}
\cE(\tens{G})=\cK\oplus\Span\{\matr{R}_1,\ldots,\matr{R}_q\},
\ \  \dim\cE(\tens{G})=36+q\le|\set{S}|.
\end{equation}
Equality in the last inequality makes the original matrices
$\{\matr{Q}_i:i\in\set{S}\}$ a basis of this image.
\end{lemma}
\begin{proof}
Write the compensation as $\sum_\nu\matr{M}_\nu\otimes\vect{f}_\nu$.
Its coefficient vectors have independent $\cZ$-projections.
Consequently functionals vanishing on $\cV$ recover every
$\matr{M}_\nu$, while each $\eta\in\cV^*$ extends uniquely to
annihilate all $\vect{f}_\nu$ and hence recover
$\eta\lrcorner\tens{P}$. This proves \eqref{eq:image-exact},
the reverse inclusion being immediate.
The independence assumptions then give the direct sum in
\eqref{eq:image-criterion}. Finally, the inclusion
$\cE(\tens{G})\subseteq\Span\{\matr{Q}_i:i\in\set{S}\}$
gives the dimension and basis assertions.
\end{proof}
The next lemma gives a second way to put every original matrix
in the contraction image, without requiring a zero intersection
with $\cK$.

\begin{lemma}[Saturation without a zero-intersection assumption]
\label{lem:triple-saturation}
For a symmetric adjoin as in \eqref{eq:adjoin}, let
$\tens{T}=\sum_{i=1}^N\vh_i^3$ and retain the substitution maps
and disjoint sets $\set{I}_A,\set{I}_B,\set{I}_C$ constructed in
\cite[Appendix~A, proof of Lemma~3.4]{Lovitz26:arxiv}.
Put $\ell\eqdef N-3m$ and let $\set{J}$ consist of the $\ell$ remaining indices.
For a selected pair $A,B$, set
\[
\begin{aligned}
\tens{P}&\eqdef(\matr{\Pi}_A\otimes\matr{\Pi}_B\otimes\matr{\Pi}_C)\tens{T},\ \ \ 
d_C\eqdef\dim\Ess_C(\tens{P}),\ \
\tens{G}_{AB}\eqdef(\matr{\Pi}_A\otimes\matr{\Pi}_B\otimes\id)\tens{T},\\
\matr{Q}_i&\eqdef(\matr{\Pi}_A\vh_i)\otimes(\matr{\Pi}_B\vh_i),\ \ \ 
\cW_{AB}\eqdef\Span\{\matr{Q}_i:i\in\set{I}_C\sqcup\set{J}\},
\cE_{AB}\eqdef\cE(\tens{G}_{AB}).
\end{aligned}
\]
Put $\rho_C\eqdef\dim\Span\{\matr\Pi_C\vh_i:i\in\set J\}$.
Then
\begin{equation}\label{eq:H345-source-span}
\begin{split}
\cW_{AB}=\cE_{AB}+\Span\{\matr Q_i:i\in\set J\},\ \
\dim(\cW_{AB}/\cE_{AB})\le \ell-\rho_C\le \ell-d_C.
\end{split}
\end{equation}
In particular $d_C=\ell$ implies $\cW_{AB}=\cE_{AB}$.
If, in addition, the residual matrix slice space of $\tens{P}$ for
this pair has zero intersection with $\cK$, the $m+\ell$ original
matrices form a basis of $\cW_{AB}=\cE_{AB}$. The assertions hold after any permutation
of the original modes.
\end{lemma}
\begin{proof}
By the cited construction, the $m$ vectors indexed by $\set{I}_C$
have independent $\cZ$-projections and span $\ker\matr{\Pi}_C$.
The residual expression gives $d_C\le\rho_C\le \ell$.
To prove the source-space equality, let a matrix functional annihilate
both $\cE_{AB}$ and the residual matrices $\matr Q_i$ with $i\in\set J$.
Applying it to the complete coefficient identity leaves a relation
among the vectors $\vh_i$ with $i\in\set I_C$. Their independent $\cZ$-projections
force all coefficients to vanish. The functional therefore annihilates
$\cW_{AB}$, proving the equality.
Let $\matr{H}$ have as columns the $m+\ell$ complete vectors $\vh_i$
indexed by $i\in\set{I}_C\sqcup\set{J}$, and let $\matr{Q}$
have the matching vectorized matrices $\matr{Q}_i$ as columns.
Zero matrix columns are retained. Then
$\rank\matr{H}=m+\rho_C,\  
\cE_{AB}=\im(\matr{Q}\matr{H}^\top)$ and $\cW_{AB}=\im\matr{Q}.$ 
Sylvester's rank inequality
\cite[Theorem~3.6]{wang2023lower} gives
\[
\dim\cW_{AB}-\dim\cE_{AB}
=\rank\matr{Q}-\rank(\matr{Q}\matr{H}^\top)
\le \ell-\rho_C\le \ell-d_C.
\]
If $d_C=\ell$, this proves saturation. If the stated intersection
also vanishes, \eqref{eq:image-exact}, with $m$ in place of $36$,
gives $\dim\cE_{AB}=m+\ell$, so the original generators form a basis.
\end{proof}

We call a selected pair $XY$ \emph{saturated} when
$\cW_{XY}=\cE_{XY}$. For the double copy, a compatible $(3,3,3)$
residual has
$3\le N-108\le3$ by \eqref{eq:residual}. Hence $N=111$ and $\ell=3$,
so \cref{lem:triple-saturation} gives saturation for every selected
pair, without any assumption on intersection with $\cK$.

\subsection{Contraction-image obstructions}

\begin{lemma}[A local good-quotient obstruction]
\label{lem:triple-good}
Consider an original-source expression of the form
\eqref{eq:G-original}, for any number of copies, with substitution
maps fixing $\cV$. For a local matrix $\matr M$, let
$[\matr M]_j\eqdef\matr M+\cK_j$, and let
$\overline{\matr p}_j$ denote restriction to the $jj$ block followed
by passage to this quotient.
Suppose
\[
(\overline{\matr p}_j\otimes\id)\tens{G}
=[\vu\otimes\vv]_j\otimes\vw,\ \
\vu\otimes\vv\otimes\vw\in\cA_j,
\]
and every original local matrix
$\matr{M}_i\eqdef(\matr{p}_j\matr{\Pi}_A\vh_i)
\otimes(\matr{p}_j\matr{\Pi}_B\vh_i)$ 
belongs to $\cK_j+\FF(\vu\otimes\vv)$. Then $\vu,\vv$ are
proportional. The local-matrix hypothesis holds, in particular,
if the original matrices span exactly $\cE(\tens{G})$.
\end{lemma}
\begin{proof}
Suppose $\vu,\vv$ are independent. By \cref{lem:matrix-rigidity},
each $\matr{M}_i$ with nonzero quotient class is exactly
$\lambda_i\vu\otimes\vv$, with $\lambda_i\ne0$.
Set $\lambda_i\eqdef 0$ for the other sources.
The complete quotient identity, in which the coefficients are
the original vectors in $\cV\oplus\cZ$, gives
$\vw=\sum_i\lambda_i\vh_i.$
Applying $\matr{p}_j\matr{\Pi}_A$ and
$\matr{p}_j\matr{\Pi}_B$ gives
$\vw\in\FF\vu$ and $\vw\in\FF\vv,$ 
because both maps fix $\vw\in\cV_j$ and each contributing source
has its two local factors on the indicated lines.
Thus $\vw=0$, contrary to membership in $\cA_j$.
Finally, the quotient identity implies that the local projection
of $\cE(\tens{G})$ is contained in
$\cK_j+\FF(\vu\otimes\vv)$. This proves the last assertion.
\end{proof}

\begin{lemma}[Saturated pairs and local rank-one quotients]
\label{lem:triple-saturated-pairs}
For any number of physical blocks, suppose at least two original
mode pairs satisfy $\cW_{XY}=\cE_{XY}$.
No pure block can have the same three two-mode quadratic quotients
as a rank-one representative $\tens R_j\in\cA_j$.
In particular, no pure block has CP rank one.
\end{lemma}
\begin{proof}
The factors of $\tens R_j$ are not all proportional, since
$\cA_j$ contains no cube. Hence at least two of its three pairs
are good, and one of these pairs is saturated. Mixed affine
components and compensation vanish in its local quadratic quotient,
so the pure-block identity is also the complete quotient identity.
Now \cref{lem:triple-good} gives a contradiction.
\end{proof}

\begin{lemma}[Fixed-source image obstructions]
\label{lem:triple-source-counting}
For a two-mode tensor obtained from a fixed minimal symmetric decomposition
of any number of copies by the original substitution maps, let
$\cE$ be its complete contraction image, let $\cW$ be its original source
span, set $e\eqdef\dim\cE$, and let $m$ be the number
of nonzero source matrices. For a selected pair $XY$ and block $j$,
the \emph{local quotient} is the image of
$(\matr p_j\otimes\matr p_j)\cE$ in
$(\cV_j\otimes\cV_j)/\cK_j$. If it is a nonzero line generated by
the class of a rank-one affine representative, call block $j$
$XY$-good or $XY$-bad according as that representative's $XY$
mode pair is good or bad in \cref{def:good}.
\begin{enumerate}
\item If two $XY$-good block rows have no off-diagonal output and their
local quotients are represented by $\cK_j+\FF(\vu_j\otimes\vv_j)$, then $m\ge e+2$.
\item If $\cE$ is block diagonal, two blocks are $XY$-good, and
$\dim(\cW/\cE)\le1$, then $m\ge e+3$.
\end{enumerate}
\end{lemma}
\begin{proof}
For (i)--(ii), suppose the asserted bound fails. Equality $\cW=\cE$
is excluded by \cref{lem:triple-good}. The dimension bounds give
\[
\cW=\cE\oplus\FF\matr\Omega,\ \ 
\matr Q_i=\matr E_i+t_i\matr\Omega,
\ \ \sum_{t_i\ne0}t_i\vh_i=0.
\]
Here $\matr E_i\in\cE$. Call $i$ exceptional if $t_i\ne0$ and normal
otherwise. For exceptional $i$, set $\matr D_i\eqdef\matr Q_i/t_i$. If blocks
$q,r$ are $XY$-good, both
$\matr\Omega_{qq},\matr\Omega_{rr}$ lie outside their local images:
otherwise every local source matrix on that block would remain in its
local contraction image, contrary to \cref{lem:triple-good}.
Thus every $\matr D_i$ has nonzero $qq,rr$ blocks. Rank one forces
nonzero $qr,rq$ blocks, fixed by $\matr\Omega$. These fix the local
factor lines. The $qq$ blocks are proportional, and their difference
lies in the local image; since neither belongs to that image, the
proportionality scalar is one. Consequently all exceptional
$\matr D_i$ have the same nonzero $q$ row and a common complete
second-factor line, say $\matr\Pi_Y\vh_i=\beta_i\vv$.

In (i), we have $m=e+1$, and the source matrices form a basis. But
\begin{equation}\label{eq:triple-exceptional-dependence}
\sum_{t_i\ne0}\frac{t_i}{\beta_i}\matr Q_i
=\left(\sum_{t_i\ne0}t_i\matr\Pi_X\vh_i\right)\otimes\vv=0,
\ \  t_i/\beta_i\ne0,
\end{equation}
a contradiction. The transposed argument applies to a common first
factor.

In (ii), block diagonality fixes the entire $q$ column as well.
A rank-one matrix with nonzero $qq$ block is determined by that row
and column, so every $\matr D_i$ equals one matrix $\matr D$.
The normal sources therefore span $\cE$: at least $e$ are normal
and at most two are exceptional. One exceptional source would give
$\vh_i=0$. For two, their complete relation gives
\[
\vh_2=\lambda\vh_1,\ 
\matr Q_2=\lambda^2\matr Q_1,\ 
t_2=\lambda^2t_1,\  1+\lambda^3=0.
\]
Hence $\vh_1^3+\vh_2^3=0$, contrary to minimality.
\end{proof}

\section{The double-copy argument}
\label{sec:small}
\subsection{Low-rank reduction}
\begin{corollary}[A block-diagonal image with one extra term]
\label{lem:block38}
Suppose the same $\tens G$ can be written as two pure rank-one
representatives $\tens P_j\in\cA_j$ plus compensation in $\cK$
with independent $\cZ$-projections. If the selected pair $XY$ is good
for both representatives, its original-source expression cannot have at most
$39$ terms from a minimal symmetric decomposition.
\end{corollary}
\begin{proof}
By \cref{lem:image-criterion}, the image is block diagonal of
dimension $38$. Both blocks are $XY$-good, so the two block rows require at least $38+2=40$
sources by \cref{lem:triple-source-counting}(i).
\end{proof}

\begin{proposition}[Low-rank reduction]
\label{prop:first-lower}
If a minimal symmetric decomposition of $\tens{V}_2$ has $N\le111$ terms,
then $N=111$ and every residual obtained as above has CP rank three.
\end{proposition}
\begin{proof}
By \eqref{eq:residual}, we have $2\le r(\tens{P})\le N-108\le3$.
If $r(\tens{P})=2$, \cref{lem:rank-two} gives two pure
rank-one representatives. Choose a pair that is good for both representatives.
By \eqref{eq:G-compensation}, the same $\tens{G}$ is their sum
plus compensation with independent $\cZ$-projections.
Its original-source expression has $|\set{S}|\le N-72\le39$ terms,
contradicting \cref{lem:block38}. Hence $r(\tens{P})=3$,
and $3\le N-108$ forces $N=111$.
\end{proof}

\begin{remark}[The additional low-rank reduction]
\label{rem:reduction-literature}
The general bound \cite[Theorem~3.2]{Lovitz26:arxiv} gives only
$s(\tens{V}_2)\ge110$. The preceding proposition excludes $110$
and leaves compatible CP-rank-three residuals at $N=111$.
These residuals need not be symmetric, so symmetric rank equality
at CP rank three \cite[Corollary~4.2]{friedland2016remarks}
does not settle the remaining case.
\end{remark}

\subsection{Mode restrictions and the \texorpdfstring{$(3,3,2)$}{(3,3,2)} case}
We first study tensors of CP rank three in $\cA$. The mode and support
restrictions below use only the affine block equations; the original
$111$-term decomposition hypothesis is imposed only where compatibility is specified.

\begin{lemma}[Mode-space restrictions]\label{lem:modes}
Every mode space of a tensor $\tens{P}\in\cA$ of CP rank three contains
a nonzero pure direction in each block. Each of its block
projections has dimension at most two. Its multilinear rank is
either a permutation of $(3,3,2)$ or $(3,3,3)$.
\end{lemma}
\begin{proof}
Let $\matr{L}_r:\cV_1\oplus\cV_2\to\FF^2$ apply row $r$
of the selector separately on the two blocks. Equation \eqref{eq:six}
and the mixed-block conditions give
$(\matr{L}_1\otimes\matr{L}_2\otimes\matr{L}_3)\tens{P}
=-2(\vect{e}_1^3+\vect{e}_2^3).$ 
In fact, contracting two modes by distinct rows supported on block
$j$ gives a vector $\vx_j\in\cV_j$ with
$\varphi_t(\vx_j)=-2$ for the remaining row $t$. Therefore
\[
\Ess_m(\tens{P})\cap\cV_j\ne0\ (m=A,B,C;\ j=1,2).
\]
If any projected three-column factor matrix had rank three,
the corresponding mixed equation would put all opposite-block
factor pairs in $\cK_{3-j}$, making
$\tens{P}_{3-j,3-j,3-j}\in\J(\cK_{3-j})$. Thus
\[
\dim\matr{p}_j\Ess_m(\tens{P})\le2,
\ \  2\le\dim\Ess_m(\tens{P})\le3.
\]
Suppose two modes have dimension two. Their spaces are sums of pure
lines, so write
\[
\tens{P}=\sum_{j,k=1}^2\va_{jk}\otimes\vb_j\otimes\vc_k,
\ \  \vb_j,\vc_j\in\cV_j.
\]
The pure-block components are rank-one affine representatives. The $(2,1,1)$
and $(1,2,2)$ equations, and their same-block pair exclusions, imply
$\va_{jj}\in\cV_j$. A nonzero block component of $\va_{12}$ would
give either $\vb_1^2\in\cK_1$ or $\vc_2^2\in\cK_2$, by symmetry
of the appropriate mixed matrix. Both are forbidden by \cref{lem:local}.
Similarly $\va_{21}=0$. Hence $r(\tens{P})=2$, a contradiction.
Only $(3,3,2)$ and $(3,3,3)$, up to permutation, remain.
\end{proof}

\begin{lemma}[Exclusion of \((3,3,2)\)]\label{lem:332}
No residual compatible with the original $111$-term symmetric decomposition
has multilinear rank $(3,3,2)$, in any order.
\end{lemma}
\begin{proof}
Take $d_C=2$ and put $\set U_C\eqdef\Ess_C(\tens P)$. By
\cref{lem:modes}, we have
$\set U_C=\FF\vc_1\oplus\FF\vc_2$ with
$0\ne\vc_j\in\cV_j$, so
$\tens P=\matr M_1\otimes\vc_1+\matr M_2\otimes\vc_2$.
Each pure affine block can be written, in the relevant mode ordering,
as $\vc_j\otimes\matr C_j$. Thus
\cref{lem:local} gives $\vc_j^2\notin\cK_j$, and symmetry implies
$\vx\otimes\vc_j\in\cK_j\Rightarrow\vx=0$.
Applied coefficientwise to the mixed equations, this makes both
$\matr M_j$ block diagonal.

Choose a functional nonzero on the three $C$ factors of a minimal
CP decomposition. The $A,B$ factor lists are independent, so
\cref{lem:factor-calculus} gives a rank-three block-diagonal slice.
Its column and row spaces are the mode spaces. Both therefore
split with the same local dimensions $1+2$, by \cref{lem:modes}.
The common one-dimensional block consequently has a nonzero rank-one
pure component. But $d_A=d_B=3$ saturates $BC,AC$ by
\cref{lem:triple-saturation}, contradicting
\cref{lem:triple-saturated-pairs}.
\end{proof}

Every remaining compatible residual has multilinear rank $(3,3,3)$.
The following affine classification applies to any tensor of CP rank three
and this multilinear rank. A mode space $\set U$ is called
\emph{split} when $\set U=\bigoplus_j(\set U\cap\cV_j)$.
For two blocks, a mode whose block projection dimensions are $1,2$
or $2,1$ is split. If the dimensions are $2,2$, the three-dimensional
mode space has one genuinely mixed direction in addition to
its two pure directions. The next subsection classifies this
nonsplit case using only the affine equations.

\subsection{Residuals with no split mode}\label{sec:bridged}
In this section every mode projects with dimension two on both
blocks. Fix a minimal CP decomposition
$\tens{P}=\sum_{i=1}^3\va_i\otimes\vb_i\otimes\vc_i$.
Let $\matr{A}_j,\matr{B}_j,\matr{C}_j$ be the three-column
matrices of the projections on block $j$.
All six matrices have rank two. Choose nonzero kernel vectors
\begin{equation}\label{eq:kernels}
\vect{\kappa}_A^j\in\ker\matr{A}_{3-j},\ \ 
\vect{\kappa}_B^j\in\ker\matr{B}_{3-j},\ \ 
\vect{\kappa}_C^j\in\ker\matr{C}_{3-j},
\end{equation}
and let $\set{S}_m^j$ be their supports. Superscript $j$ records
the block containing the resulting pure vector, not the block
whose projection is annihilated.

\subsubsection{Quotient relations and full supports}
For example, the $(2,1,1)$ equation gives
$\sum_i\va_i^{[3-j]}\otimes
[\vb_i^{[j]}\otimes\vc_i^{[j]}]_j=0$.
For a linear map $\matr R$ and a vector space $\set H$, exactness over a field gives
$\ker(\matr R\otimes\id_{\set H})=(\ker\matr R)\otimes\set H$.
Taking $\matr R=\matr{A}_{3-j}$, whose kernel is one-dimensional, proves
\begin{equation}\label{eq:quotient}
[\vb_i^{[j]}\otimes\vc_i^{[j]}]_j
=(\kappa_A^j)_i\theta_A^{[j]}
\ (1\le i\le3),
\end{equation}
and its cyclic counterparts, where $\theta_A^{[j]}$ is a nonzero
class modulo $\cK_j$. The class is nonzero, since otherwise the pure
block would belong to $\J(\cK_j)$.
The vector
$\widehat{\va}_j\eqdef\sum_i(\kappa_A^j)_i\va_i^{[j]}$ is a nonzero
pure-block vector: the original three first-mode factors are
independent. At every nonzero kernel coordinate, we have
\begin{equation}\label{eq:quotient-rep}
\frac{\widehat{\va}_j}{(\kappa_A^j)_i}
\otimes\vb_i^{[j]}\otimes\vc_i^{[j]}\in\cA_j.
\end{equation}
Indeed, replacing every pair by its common quotient class in
$\tens{P}_{jjj}$ changes the tensor by an element of
$\cV_j\otimes\cK_j$ and gives exactly
\eqref{eq:quotient-rep}. Its first factor is nonzero because
$\sum_i(\kappa_A^j)_i\va_i$ is nonzero and its opposite
projection vanishes. These are single-block affine representatives.

\begin{lemma}[Common factor planes and full supports]\label{lem:full-support}
Let $(\vb_i,\vc_i)_{i=1}^m$ be pairs in $\cV_j$, with both
factor spans of dimension two. Suppose their classes modulo
$\cK_j$ are nonzero and proportional, and one pair extends to a
rank-one tensor in $\cA_j$. Then both factor spans equal a plane
$\cW$ for which $\cK_j\cap\Sym^2\cW$ contains a rank-two matrix.
In particular, in the present kernel-support setting, at most one
of the three supports in a fixed block is full.
\end{lemma}
\begin{proof}
Proportionality modulo $\cK_j$ extends every pair to a rank-one
affine representative, after rescaling its remaining factor.
If some $\vb_i,\vc_i$ were independent, \cref{lem:matrix-rigidity}
would make every pair-product proportional to
$\vb_i\otimes\vc_i$, contrary to the two-dimensional factor spans.
Hence $\vb_i\parallel\vc_i$ for all $i$, and the two spans coincide.
Choose two independent directions. Their normalized pair-products
have equal quotient classes, so their nonzero difference belongs to
$\cK_j\cap\Sym^2\cW$ and has rank two.

For a full support, \eqref{eq:quotient}--\eqref{eq:quotient-rep}
supply the hypotheses. Two full supports would therefore identify
all three projected mode planes with the same $\cW$. Then
$\tens P_{jjj}\in\cA_j\cap\cW^{\otimes3}$, contradicting
\cref{lem:binary,rem:symmetrization}.
\end{proof}

\subsubsection{The finite support classification}
Set $[3]\eqdef\{1,2,3\}$ and
\[
\epsilon_j(i)\eqdef\#\{m\in\{A,B,C\}:i\in\set{S}_m^j\},
\ \  \bar j\eqdef 3-j.
\]
Write $\vx_{A,i}^{[j]}\eqdef\va_i^{[j]}$; write
$\vx_{B,i}^{[j]}\eqdef\vb_i^{[j]}$; and write
$\vx_{C,i}^{[j]}\eqdef\vc_i^{[j]}$.
\begin{proposition}[Support classification]\label{prop:support-patterns}
Up to simultaneous permutation of the three terms, permutation
of the modes, and exchange of the blocks, the supports are
\begin{equation}\label{eq:patterns}
\begin{array}{c|c|c}
&(\set{S}_A^1,\set{S}_B^1,\set{S}_C^1)&
(\set{S}_A^2,\set{S}_B^2,\set{S}_C^2)\\ \hline
\mathrm{I}&(\{1\},\{1\},\{1\})&(\{2\},\{2\},\{2\})\\
\mathrm{II}&(\{1\},\{1\},\{1\})&
(\{2,3\},\{2,3\},\{2,3\}).
\end{array}
\end{equation}
\end{proposition}
\begin{proof}
We first establish five rules for the six nonempty supports:
\begin{enumerate}
\item $\#\{m:\set{S}_m^j=[3]\}\le1$;
\item $\epsilon_j(i)\in\{0,1,3\}$;
\item If $\epsilon_j(i)=1$ and $i\in\set{S}_m^j$, then
$\vx_{m,i}^{[j]}=0$ and $\set{S}_m^{\bar j}=\{i\}$;
\item If $\set{S}_m^j=\{i\}$, then $\epsilon_j(i)=3$, while
$i\notin\bigcup_{n\ne m}\set{S}_n^{\bar j}$ and
$\set{S}_m^{\bar j}\ne\{i\}$;
\item $\min_{j,m}|\set{S}_m^j|=1$.
\end{enumerate}
Throughout, activity in mode $A$ means that the $B,C$ pair
has nonzero quotient class; use the cyclic convention for the
other modes. Rule (i) is \cref{lem:full-support}.

For (ii), suppose two modes are active. Their two rank-one
representatives in \eqref{eq:quotient-rep} show that all three
projected factors are nonzero and selector-dense. If the third
mode were inactive, the associated pair would belong to $\cK_j$.
A nonzero symmetric rank-one pair has proportional factors;
its square would belong to $\cK_j$, contradicting the nonzero
off-diagonal entries of the selector image. Exactly two active modes
are therefore impossible.

For (iii), take $A$ to be the only active mode. Then
\[
\va_i^{[j]}\otimes\vb_i^{[j]},\ \va_i^{[j]}\otimes\vc_i^{[j]}\in\cK_j,
\ \  \vb_i^{[j]},\vc_i^{[j]}\text{ are selector-dense}.
\]
If $\va_i^{[j]}\ne0$, symmetry of either matrix makes its two
factors proportional and puts the square of a dense factor
in $\cK_j$, a contradiction. Thus $\va_i^{[j]}=0$.
The kernel of a rank-two three-column matrix with column $i$
zero is exactly $\FF\vect{e}_i$, proving the singleton assertion.

For (iv), suppose $\set{S}_A^j=\{i\}$. Then
$\va_i^{[\bar j]}=0$ while $\va_i^{[j]}\ne0$, since a full CP factor
is nonzero. Activity in $A$ makes $\vb_i^{[j]},\vc_i^{[j]}$ dense.
If $B$ or $C$ were inactive, its pair in $\cK_j$ would again
force the square of one of these dense factors into $\cK_j$.
Thus $i$ is in all three supports on side $j$.
On the opposite side the $A$ factor is zero, so both pairs
containing it have zero quotient class. Hence
$i\notin\set{S}_B^{\bar j}\cup\set{S}_C^{\bar j}$.
Finally $\set{S}_A^{\bar j}\ne\{i\}$, since that singleton
would also force $\va_i^{[j]}=0$.

For (v), suppose no projected column vanishes. No support
can be a singleton, since a singleton kernel is a zero column.
Rule (iii) rules out activity one, and (ii) rules out activity
two. Hence all three supports on a given side coincide; denote their
common support by $\set S^j$. They have size two by (i) and
nonemptiness. Write $\set{S}^1=\{p,q\}$. The opposite factor matrices have
proportional nonzero columns $p,q$ in every mode, so their
pair-products have the same zero or nonzero status modulo
$\cK_2$. Thus $p,q$ have equal opposite activity. The common
opposite support has size two, so it must also be $\{p,q\}$.
For the remaining term index $\ell$ and each block $j\in\{1,2\}$,
the three factors in block $j$ are proportional to a nonzero $\vw_j$ with
$\vw_j^2\in\cK_j$. In every block component of its full CP
term at least two positions belong to the same block; their
pair is a multiple of $\vw_j^2$. Therefore the entire term is
in $\J_{\cV}(\cK)$. Deleting it leaves a two-term tensor in
$\cA$, of CP rank exactly two by the affine minimum in
\cref{prop:ordinary}. By \cref{lem:rank-two} its two terms must
be pure, contradicting the assumption of no zero projected
column. A zero column exists and yields a singleton support.

We now solve these rules.
By rule (v), normalize $\set{S}_A^1=\{1\}$; rule (iv) gives
$1\in\set{S}_B^1\cap\set{S}_C^1$.
Suppose $2\in\set{S}_B^1$. Rules (ii)--(iv) successively imply
\begin{align*}
&2\notin\set{S}_C^1,\ \  \set{S}_B^2=\{2\},
\ \  2\in\set{S}_A^2\cap\set{S}_C^2,\\
&3\in\set{S}_C^1
\ \Longrightarrow\ \set{S}_C^2=\{3\},\ \text{impossible},\\
&\set{S}_C^1=\{1\},\ \ 
1\notin\set{S}_A^2\cup\set{S}_C^2,\ \
\set{S}_C^2\ne\{2\},\ \  \set{S}_C^2=\{2,3\},\ \
3\notin\set{S}_A^2,\ \  \set{S}_A^2=\{2\}.
\end{align*}
Here the exclusions involving index $1$ are obtained by applying
rule~(iv) to both singleton supports $\set S_A^1=\{1\}$ and
$\set S_C^1=\{1\}$.
The final singleton contradicts $2\in\set{S}_B^1$ by rule (iv).
The same argument under permutations gives
$\set{S}_A^1=\set{S}_B^1=\set{S}_C^1=\{1\}$. 
All opposite supports lie in $\{2,3\}$. If one is $\{i\}$,
rule (iv) puts $i$ in all three. The other index cannot have activity
one, two, or three: these contradict (iii), (ii), or that singleton,
respectively. This is pattern I. Without a singleton, all three
opposite supports are $\{2,3\}$, giving pattern II.
\end{proof}

\subsection{A shared two-block affine form}
The preceding support calculation yields a statement independent of
the original source count. It will be used for both direct sums.

\begin{lemma}[Affine form of CP rank three and multilinear rank \((3,3,3)\)]
\label{lem:H345-R333}
Let $\tens Q$ belong to the affine space of two physical blocks
$q,r$, with
$r(\tens Q)=3$ and $\mrank(\tens Q)=(3,3,3).$ 
Put $d^{\mathrm{proj}}_{X,j}\eqdef\dim\matr p_j\Ess_X(\tens Q)$.
Then at least one side block has the form
\begin{equation}\label{eq:H345-side-general}
\tens Q_{jjj}=\tens R_j+\lambda_j\vw_j^3,\ \ 
r(\tens R_j)=1,\ \tens R_j\in\cA_j,\ 
\vw_j^2\in\cK_j,
\end{equation}
where $\lambda_j=0$ is allowed. If, in addition, we have
$d^{\mathrm{proj}}_{X,q}=d^{\mathrm{proj}}_{X,r}=2$ for all three modes, then, after interchanging
$q,r$ and the three terms if necessary, we have
\begin{equation}\label{eq:H345-Type-I}
\tens Q=\tens R_q+\tens R_r+
(\alpha_q\vw_q+\alpha_r\vw_r)\otimes
(\beta_q\vw_q+\beta_r\vw_r)\otimes
(\gamma_q\vw_q+\gamma_r\vw_r),
\end{equation}
where all six scalars are nonzero, the inclusions
$\vw_j^2\in\cK_j$ hold, and each $\tens R_j\in\cA_j$ has rank one. In
particular
\begin{equation}\label{eq:H345-side-quotient}
\tens Q_{jjj}=\tens R_j+\lambda_j\vw_j^3
\ (j=q,r),\ \ \lambda_j\ne0.
\end{equation}
\end{lemma}
\begin{proof}
Suppose first that one mode splits, with a line $\FF\va$ in
block $q$. By \cref{lem:modes}, every projected mode has dimension
at most two. The mixed equations and $\va^2\notin\cK_q$ give a
block-diagonal contraction $\matr M_q\oplus\matr N_r$, where
$\matr N_r\in\cK_r$ and $1\le\rank\matr M_q\le2$.
If $\rank\matr M_q=1$, the $q$ pure block itself is a rank-one
representative.

If $\rank\matr M_q=2$ and $\matr N_r\ne0$, the contraction has
rank three and $\rank\matr N_r=1$. By \cref{lem:factor-calculus},
its supports are the full remaining mode spaces; their $r$ parts
are the factor lines of $\matr N_r$. Hence
$\tens Q_{rrr}\in\cV_r\otimes\cK_r$, impossible in $\cA_r$.
Thus $\matr N_r=0$. Exactly two terms are active, and their two
contracted factors lie in block $q$. Only the remaining term
contributes to the nonzero $r$ pure block, which is therefore a
rank-one representative. This proves \eqref{eq:H345-side-general}
with $\lambda_j=0$ whenever a mode splits.

Otherwise every mode projects two-dimensionally onto both blocks,
by the projection bound of \cref{lem:modes}. The proofs of
\cref{lem:full-support,prop:support-patterns} use only the affine
block equations, so the supports have pattern I or II.
In pattern II, apply \cref{lem:full-support} to the two $A,B$
pairs and then the two $A,C$ pairs on the pattern-II full-support side.
Their common first-factor plane identifies all three mode planes
with a plane containing a rank-two element of $\cK_j$.
This contradicts \cref{lem:binary,rem:symmetrization}. 
In pattern I, the first two terms are pure on the two sides. The
third has, on each side, three proportional nonzero projections,
say along $\vw_j$, and the inactive quotient equations give
$\vw_j^2\in\cK_j$. Its pure-block components lie in
$\J_{\cV_j}(\cK_j)$; subtracting them leaves the first two terms in
$\cA_q,\cA_r$. This proves
\eqref{eq:H345-Type-I}--\eqref{eq:H345-side-quotient}.
\end{proof}

\subsection{Fixed-source exclusion and conclusion}\label{sec:split}
The affine classification now meets the fixed-source obstruction.

\begin{proposition}[Exclusion of compatible CP-rank-three residuals]
\label{prop:double-333}
No CP-rank-three residual is compatible with a minimal
$111$-term symmetric decomposition of $\tens{V}_2$.
\end{proposition}
\begin{proof}
By \cref{lem:modes,lem:332}, such a residual has multilinear rank
$(3,3,3)$. All three original pairs are saturated by
\cref{lem:triple-saturation}. On the other hand,
\cref{lem:H345-R333} gives a pure block
\[
\tens P_{jjj}=\tens R_j+\lambda\vw_j^3,
\ \  r(\tens R_j)=1,\ \tens R_j\in\cA_j,
\ \  \vw_j^2\in\cK_j.
\]
The term $\lambda\vw_j^3$ vanishes in each two-mode quadratic quotient. Thus
this block has the same three quotients as $\tens R_j$, contrary
to \cref{lem:triple-saturated-pairs}.
\end{proof}

\begin{remark}[Relation to symmetric additivity results]
\label{rem:symmetric-additivity-literature}
The available symmetric-rank additivity criteria
\cite{carlini2015symmetricstrassen,carlini2018ecomputability,woo2014cases}
do not apply to $\tens T_0$; the fixed-source argument proves the
required lower bound directly over both fields.
\end{remark}

\begin{proof}[Proof of \cref{thm:main} for $k=2$]
\Cref{prop:ordinary} gives CP rank $110$ and concision on a
$54$-dimensional space over either field. The single-copy symmetric rank gives
$s(\tens{V}_2)\le2s(\tens{T}_0)=112$.
Suppose a minimal symmetric decomposition has $N\le111$ terms.
By \cref{prop:first-lower}, we have $N=111$, and every associated residual
has CP rank three, contrary to \cref{prop:double-333}. Hence $s(\tens{V}_2)=112$ and
$\DeltaT_{\FF}(\tens{V}_2)=2$. The construction has rational coordinates.
\end{proof}
\section{The triple copy below \texorpdfstring{$167$}{167} terms}
\label{sec:triple}
In this section the block index ranges over $\{1,2,3\}$. Whenever
$p,q,r$ occur together, they denote the three distinct blocks; whenever
$X,Y,Z$ occur together, they denote the three distinct modes. Set
\[
\cV\eqdef\bigoplus_{j=1}^3\cV_j,\ \ 
\cZ\eqdef\bigoplus_{j=1}^3\cZ_j,\ \ 
\cK\eqdef\bigoplus_{j=1}^3\cK_j,\ \ 
\cA^{[3]}\eqdef\sum_{j=1}^3\tens{F}_j+\J_{\cV}(\cK).
\]
Thus $\dim\cV=27$ and $\dim\cK=\dim\cZ=54$.
By \cref{prop:ordinary,prop:single}, we have
\begin{equation}\label{eq:triple-initial}
r(\tens{V}_3)=165,\ \  s(\tens{V}_3)\le168,
\ \ \dim\Ess_A(\tens{V}_3)=81.
\end{equation}
It remains to prove the symmetric lower bound. In addition to its
pure-block and repeated-block conditions, every
$\tens{P}\in\cA^{[3]}$ satisfies
\begin{equation}\label{eq:triple-slots}
\tens{P}_{ijk}=0\ (i,j,k\text{ pairwise distinct}),
\ \ 
\tens{P}_{ijj}\in\cV_i\otimes\cK_j\ (i\ne j),
\end{equation}
with the analogous conditions in the other two orders.
These are conditions on the actual tensor components, not on each
term of an arbitrarily chosen affine correction.

\subsection{Original sources and low-dimensional block tools}
For a minimal symmetric decomposition $\tens{V}_3=\sum_{i=1}^N\vh_i^3$,
the original substitution maps give disjoint sets
$\set{I}_A,\set{I}_B,\set{I}_C$, each of size $54$.
Write $\set{J}$ for the remaining indices and put
\begin{equation}\label{eq:triple-source}
\va_i\eqdef\matr{\Pi}_A\vh_i,\ 
\vb_i\eqdef\matr{\Pi}_B\vh_i,\ 
\vc_i\eqdef\matr{\Pi}_C\vh_i,\ \ 
\tens{P}\eqdef\sum_{i\in\set{J}}\va_i\otimes\vb_i\otimes\vc_i
\in\cA^{[3]}.
\end{equation}
Set $\set{U}_X\eqdef\Ess_X(\tens{P})$ and
$d_X\eqdef\dim\set{U}_X$.
As in \cref{lem:modes}, contracting two modes by distinct selector
rows supported on block $j$ gives a nonzero vector in
$\set{U}_X\cap\cV_j$. Consequently
\begin{equation}\label{eq:triple-pure-intersections}
\dim(\set{U}_X\cap\cV_j)\ge1\ (X=A,B,C;\ j=1,2,3),
\ \  3\le d_X\le N-162.
\end{equation}

For pair $B,C$, denote the residual matrix slice space by
$\set{L}_{BC}(\tens{P})$ and define
\begin{equation}\label{eq:triple-EW}
\begin{split}
\kappa_{BC}\eqdef\dim\bigl(\set{L}_{BC}(\tens{P})\cap\cK\bigr),\ \
\cW_{BC}\eqdef\Span\{\vb_i\otimes\vc_i:
i\in\set{I}_A\sqcup\set{J}\}.
\end{split}
\end{equation}
The same definitions apply to the other pairs. Keeping the complete
first-mode coefficients, rather than their projections, gives
\begin{equation}\label{eq:triple-complete}
\begin{split}
\tens{G}_{BC}
\eqdef(\id\otimes\matr{\Pi}_B\otimes\matr{\Pi}_C)\tens{V}_3
=\sum_{i\in\set{I}_A\sqcup\set{J}}
\vh_i\otimes\vb_i\otimes\vc_i,\
\tens{G}_{BC}-\tens{P}
=\sum_{\nu=1}^{54}
(\vect{e}_\nu-\matr{\Pi}_A\vect{e}_\nu)\otimes\matr{M}_\nu.
\end{split}
\end{equation}
Thus, by \eqref{eq:image-exact}, its complete contraction image is
$\cE_{BC}=\cK+\set L_{BC}(\tens P)$, and
$\dim\cE_{BC}=54+d_A-\kappa_{BC}$. 
Write
\[
\set U_{X,j}\eqdef\set U_X\cap\cV_j,\ \
d^{\mathrm{proj}}_{X,j}\eqdef\dim\matr p_j\set U_X
\ (X=A,B,C;\ j=1,2,3).
\]
Mode $X$ is \emph{thin} on block $j$ when
$d^{\mathrm{proj}}_{X,j}=1$, and \emph{thick} there otherwise. A
block is \emph{thin} if at least one mode is thin on it and
\emph{fully thin} if all three modes are thin on it. As above,
a mode is \emph{split} when
$\set U_X=\set U_{X,1}\oplus\set U_{X,2}\oplus\set U_{X,3}$.
Write $n_0\eqdef|\set J|=N-162$ and
$\rho_Z\eqdef\dim\Span\{\matr\Pi_Z\vh_i:i\in\set J\}$.
The source-space identity \eqref{eq:H345-source-span} applies with
$\ell=n_0$. In particular, the equality $d_Z=n_0$ saturates the complementary pair,
without any assumption on its quadratic intersection.
The next statement uses only the affine block conditions.

\begin{lemma}[Thin contractions]\label{lem:H345-thin}
Let $\tens P\in\cA^{[3]}$ and suppose
$d^{\mathrm{proj}}_{A,p}=1$, say $\matr p_p\set U_A=\FF\va$.
If $f$ is supported on $\cV_p$ and $f(\va)=1$, then
\begin{equation}\label{eq:H345-thin-slice}
f\lrcorner\tens P=\matr M_p\oplus\matr N_q\oplus\matr N_r,
\ \  \matr N_q\in\cK_q,\ \matr N_r\in\cK_r,
\end{equation}
where $\va\otimes\matr M_p=\tens P_{ppp}$. We call $\matr N_q$ and
$\matr N_r$ the quadratic tails of this contraction. Moreover, we have
\begin{equation}\label{eq:H345-thin-support}
\operatorname{col}\matr M_p\subseteq\set U_{B,p},\ \
\operatorname{row}\matr M_p\subseteq\set U_{C,p},
\end{equation}
and the analogous inclusions hold for the two tails.
\end{lemma}
\begin{proof}
A surviving off-diagonal $B,C$ block in \eqref{eq:triple-slots}
would pair $\va$ with another vector in block $p$ to form a
matrix in $\cK_p$. Symmetry and $\va^2\notin\cK_p$, from
\cref{lem:local}, make that matrix zero. Three-distinct-block slots
also vanish. The remaining diagonal blocks outside $p$ lie in the
respective quadratic spaces. Their column and row spaces, as well as
those of $\matr M_p$, lie in the pure intersections of the
mode spaces, giving \eqref{eq:H345-thin-support}.
\end{proof}

\begin{lemma}[Quadratic support in small mode spaces]
\label{lem:H345-four-support}
Let $\tens P\in\cA^{[3]}$, set
$\set U_Z\eqdef\Ess_Z(\tens P)$ for $Z\in\{A,B,C\}$, fix distinct
modes $X,Y$, and suppose either $\dim\set U_X=4$, or $\dim\set U_X=5$
and $\set U_X$ is nonsplit. If
$0\ne\matr N\in\cK\cap(\set U_X\otimes\set U_Y)$, then
$\matr N$ is supported on a unique block $p$, with
$\dim\set U_{X,p}=2$. Moreover, we have
\begin{equation}\label{eq:H345-local-K-one}
\dim\bigl(\cK\cap(\set U_X\otimes\set U_Y)\bigr)\le1,
\ \ \rank\matr N\le2.
\end{equation}
When $\dim\set U_X=4$, necessarily
\begin{equation}\label{eq:H345-four-split}
\set U_X=\set U_{X,p}\oplus\set U_{X,q}\oplus\set U_{X,r},\ \
(\dim\set U_{X,p},\dim\set U_{X,q},\dim\set U_{X,r})=(2,1,1).
\end{equation}
Thus a nonsplit four-dimensional mode has zero quadratic intersection
with every other mode.
\end{lemma}
\begin{proof}
Every $\set U_{X,j}$ is nonzero by
\eqref{eq:triple-pure-intersections}.  A diagonal block
$\matr N_j\ne0$ has column space in $\set U_{X,j}$.  If this space were a
line, symmetry would make $\matr N_j$ a scalar multiple of the
square of its selector-dense generator, contradicting
\cref{lem:local}.  Hence $\dim\set U_{X,j}\ge2$.  The sum of the three pure-intersection dimensions is at most four:
this is immediate in dimension four and follows from nonsplitting
in dimension five. Thus only one block can have dimension two;
in dimension four this gives \eqref{eq:H345-four-split}. 
The two-dimensional plane $\set B\eqdef\set U_{X,p}$ contains a selector-dense
vector: each selector coordinate is nonzero on $\set B$ by the same
two-coordinate contractions used in
\eqref{eq:triple-pure-intersections}, and over the infinite fields
$\RR,\CC$ three proper kernels cannot cover $\set B$. A symmetric
matrix whose column space lies in $\set B$ belongs to $\Sym^2\set B$.

We claim that
\begin{equation}\label{eq:triple-two-plane-K}
\dim(\cK_p\cap\Sym^2\set B)\le1.
\end{equation}
Let
$\matr S\eqdef\matr L|_{\set B}$. If $\rank\matr S=2$, the map
$\matr Q\mapsto\matr S\matr Q\matr S^\top$ is injective on
$\Sym^2\set B$. Its image for $\matr Q\in\cK_p$ is diagonal. A
two-dimensional space of diagonal matrices whose images lie in the
plane $\matr S(\set B)$ would force $\matr S(\set B)$ to be a coordinate plane, which
cannot contain a vector with three nonzero coordinates. 
If $\rank\matr S=1$, choose a basis $\vb,\vw$ of $\set B$ with $\matr S\vb$
selector-dense and $\matr S\vw=0$. Write
\[
\matr Q=u\vb^2+v(\vb\otimes\vw+\vw\otimes\vb)+t\vw^2.
\]
Diagonality of $\matr S\matr Q\matr S^\top=u(\matr S\vb)^2$ gives $u=0$.
If the intersection had dimension two, it would contain a nonzero
multiple of $\vw^2$, contradicting the injectivity of
$\matr Q\mapsto\matr Q\matr L^\top$ on $\cK_p$ from
\cref{lem:faithful}.  This proves
\eqref{eq:H345-local-K-one}; the rank bound follows from the
two-dimensional column space. The proof of
\eqref{eq:triple-two-plane-K} applies to any plane $\set B$ containing
a selector-dense vector.
\end{proof}

\begin{lemma}[A common functional produces a quadratic slice]
\label{lem:triple-functional}
For $\tens{P}\in\cA^{[3]}$, put
$\Lambda_j\eqdef\im\bigl((\matr{p}_j|_{\set{U}_A})^*\bigr)
\subseteq\set{U}_A^*.$ 
Contraction induces an injection
\begin{equation}\label{eq:triple-functional}
\Lambda_1\cap\Lambda_2\cap\Lambda_3
\hookrightarrow\set{L}_{BC}(\tens{P})\cap\cK.
\end{equation}
In particular, if $d_A=4$ and $d^{\mathrm{proj}}_{A,j}=2$ for all $j$, then
$\set L_{BC}(\tens P)\cap\cK\ne0$.
\end{lemma}
\begin{proof}
A common functional $\phi$ has three extensions to $\cV$, each
supported on a different block. All give the same slice
$\matr{M}\eqdef\phi\lrcorner\tens{P}$. For an off-diagonal matrix
block $(i,j)$, use the extension supported on the third block;
\eqref{eq:triple-slots} gives $\matr{M}_{ij}=0$.
For a diagonal block $(j,j)$, use an extension supported on a
different block; the repeated-block condition gives
$\matr{M}_{jj}\in\cK_j$. Thus $\matr{M}\in\cK$.
Contraction from the dual of the mode space is injective,
which proves \eqref{eq:triple-functional}.
For the last assertion, write
$\set U_A=\Span\{\va_1,\va_2,\va_3,\vx\}$ with $\va_j\in\cV_j$.
For each $j$, the vectors $\va_j,\matr p_j\vx$ are independent.
The functional taking values $0,1$ on these two vectors is supported
on block $j$ and restricts to the same nonzero functional on
$\set U_A$. Hence the intersection on the left is nonzero.
\end{proof}

\begin{lemma}[Affine residuals of CP rank four]
\label{lem:triple-affine-four}
Let $\tens Q\in\cA^{[3]}$ satisfy
$r(\tens Q)=4$ and $\mrank(\tens Q)=(4,4,4)$.
\begin{enumerate}
\item Some pure block has CP rank one. If a thin block has pure-block CP rank
two, its thin contraction has no quadratic tails and the other two
pure blocks have CP rank one.
\item If the pure blocks $q,r$ both have rank at least two, then
\begin{equation}\label{eq:H345-four-peeling}
\tens Q=\tens Q_{ppp}\oplus\tens Q_{qr},\ \ 
r(\tens Q_{ppp})=1,\ 
r(\tens Q_{qr})=3,\ \mrank(\tens Q_{qr})=(3,3,3).
\end{equation}
Every mode of $\tens Q_{qr}$ projects two-dimensionally onto both
$q$ and $r$.
\end{enumerate}
\end{lemma}
\begin{proof}
Each mode is spanned by three pure lines and one extra vector, so
$d^{\mathrm{proj}}_{X,j}\le2$. In a minimal four-term CP decomposition every factor list
is a basis, so \cref{lem:factor-calculus} identifies the number of
active terms with the contraction rank.
Consider a thin block $p$ with pure-block CP rank two. By
\cref{lem:H345-thin}, its slice is
$\matr M_p\oplus\matr N_q\oplus\matr N_r$, where
$\rank\matr M_p=2$ and $\matr N_j\in\cK_j$.
Up to exchanging $q,r$, the nonzero-tail possibilities are
$(\rank\matr N_q,\rank\matr N_r)=(1,1),\ (2,0),\ (1,0).$ 
In the first case all four terms are active. A rank-one tail
$\lambda\vw_j^2\in\cK_j$, for $j\in\{q,r\}$, forces all local factors
onto $\FF\vw_j$, putting the $j$th pure-block component in
$\J_{\cV_j}(\cK_j)$, a contradiction. In the second case
all four terms have zero $r$ projection in the two contracted modes,
contrary to their nonzero pure intersections. In the third case the
three active terms vanish on $r$ and project onto $\FF\vw$ on $q$.
Since $\vw^2\in\cK_q$, it follows that $\matr L\vw$ is a nonzero coordinate
vector. On $q\oplus r$, choose one selector row annihilating it and
the other two rows in the other modes. Only one term survives, but
\eqref{eq:all-copy-selector} has rank two. Thus
\begin{equation}\label{eq:H345-rank2-tail-zero}
\rank\matr M_p=2\ \Longrightarrow\ 
\matr N_q=\matr N_r=0.
\end{equation}
Exactly two terms are active and have both matrix factors pure in
$p$. Projection onto $q\oplus r$ leaves an affine tensor of CP rank two;
\cref{lem:rank-two} separates it into two rank-one pure blocks.

If no block were thin, all modes would be nonsplit and all their
projection dimensions would be two. Then
\cref{lem:triple-functional} would give a nonzero quadratic slice,
contrary to \cref{lem:H345-four-support}. Hence a thin block exists.
Its pure-block CP rank is one or two, so the preceding argument proves (i).

For (ii), the rank-one block must be $p$. A thin projection on $q$
or $r$ would have pure-block CP rank two and, by (i), force the other side
block to have rank one. Thus every mode is thick on both $q,r$
and is nonsplit. None can also be thick on $p$, by the same
common-functional contradiction. Hence $p$ is thin in all modes.
Its quadratic tails vanish on the selector-dense pure side lines;
the repeated-block and three-distinct-block equations then kill
every mixed component containing $p$. Removing this independent
rank-one block subtracts one from the rank and each mode dimension,
giving \eqref{eq:H345-four-peeling}.
\end{proof}

\subsection{The \texorpdfstring{$167$}{167}-term lower bound}

\begin{proposition}[The \(167\)-term lower bound]
\label{prop:triple-H0}
For each $\FF\in\{\RR,\CC\}$, we have
$s_{\FF}(\tens{V}_3)\ge167$.
\end{proposition}
\begin{proof}
Suppose a minimal symmetric decomposition has $N\le166$ terms. Then
\[
N\in\{165,166\},\ \  n_0=N-162\le4,\ \ 
3\le d_X\le r(\tens P)\le n_0.
\]
If two modes have dimension three, they split into their three pure
lines. Every mixed coefficient vanishes: a three-distinct-block slot
is zero, and a repeated-block slot contains a selector-dense line
and hence has zero quadratic factor. Thus
\begin{equation}\label{eq:triple-two-three-modes}
d_X=d_Y=3\ \Longrightarrow\ 
\tens P=\bigoplus_{j=1}^3\tens P_j,\ \ 
r(\tens P_j)=1,\ \tens P_j\in\cA_j.
\end{equation}
Each $\tens P_j$ has at least two good pairs, so for some pair $XY$
two blocks are $XY$-good. Its image is block diagonal and satisfies
\[
\dim\cE_{XY}=57,\ \  \dim\cW_{XY}\le N-108\le58,
\ \  \dim(\cW_{XY}/\cE_{XY})\le1.
\]
This contradicts \cref{lem:triple-source-counting}(ii).
In particular $r(\tens P)=3$ is excluded.

It remains that $N=166$ and $r(\tens P)=4$, and that the mode type is
$(4,4,3)$ or $(4,4,4)$. In either case at least two pairs omit a
four-dimensional mode, so \eqref{eq:H345-source-span} gives
$\cW_{XY}=\cE_{XY}$ for each such pair $XY$. By
\cref{lem:triple-saturated-pairs}, no pure block has CP rank one.
For type $(4,4,3)$, take $d_C=3$. Its pure-line contractions are block
diagonal by \cref{lem:H345-thin}. In a minimal four-term CP decomposition
the $A,B$ factor lists are bases. A functional nonzero on all four
$C$ factors therefore gives a rank-four block-diagonal slice.
Its row and column spaces are the mode spaces, so both have
pure-intersection dimensions $(2,1,1)$. A block where $A$ is a line has pure-block CP rank one because
$C$ is also a line, a contradiction. 
For type $(4,4,4)$, \cref{lem:triple-affine-four}(i) supplies a rank-one
pure block, giving the same contradiction. Hence $N\le166$ is
impossible.
\end{proof}

\section{The \texorpdfstring{$167$}{167}-term case and completion of the triple-copy argument}
\label{sec:triple-167}

By \cref{prop:triple-H0}, it remains to rule out a minimal
$167$-term symmetric decomposition. Its original residual has five terms; we now
classify its possible mode types and eliminate them.

\subsection{Finite core and fully thin reductions}
\label{sec:triple-reductions}
Throughout these reductions the substitution maps and original decomposition
vectors stay fixed. The following core calculation depends only on
the affine block equations. 
For a four-dimensional mode space $\set U_X$, choose selector-dense
$\vu_{X,j}\in\set U_X\cap\cV_j$. Such choices exist by the
selector contractions and avoidance of three proper kernels.
Whenever a pure block has CP rank one, choose its three factor lines.
Write
\begin{equation}\label{eq:four-mode-mask}
\set U_X
=\Span\{\vu_{X,1},\vu_{X,2},\vu_{X,3},\vx_X\},
\ \ 
\vx_X\eqdef\sum_{j\in\set S_X}\vx_{X,j},
\ \ 
\vx_{X,j}\notin\FF\vu_{X,j}.
\end{equation}
Here $\set S_X$ is a nonempty subset of $\{1,2,3\}$. We encode it by the
mask
\[
\sigma_X\eqdef\sum_{j\in\set S_X}2^{j-1}\in\{1,\ldots,7\}.
\]
Thus, for example, the masks $3$ and $7$ encode $\{1,2\}$ and
$\{1,2,3\}$, respectively. We identify triples of masks under
simultaneous permutations of the three modes and of the three physical
blocks.

\begin{lemma}[The finite core closure]
\label{lem:triple-core-closure}
Use the core indices $1,2,3,4$, where $j$ denotes the pure vector
$\vu_{X,j}$ and $4$ denotes $\vx_X$. For a support set $\set S_X$, put
\[
\cB_X(j)\eqdef
\begin{cases}
\{j,4\},&j\in\set S_X,\\
\{j\},&j\notin\set S_X.
\end{cases}
\]
Starting with the $64$ core coefficients, repeatedly apply the
following rules.
\begin{enumerate}
\item If $i,j,k$ are pairwise distinct physical blocks, all coefficients
in
$\cB_A(i)\times\cB_B(j)\times\cB_C(k)$ 
vanish.
\item If the pure block $j$ has been proved rank one and at least one
mode is thin there, all coefficients in
$\cB_A(j)\times\cB_B(j)\times\cB_C(j)$ 
except $(j,j,j)$ vanish.
\item Consider a mixed slot with two occurrences of a block $j$.
If, in one of the repeated modes, every still-allowed coefficient uses
the pure index $j$, then the corresponding quadratic factor contains
a selector-dense line.  Since its product with another vector cannot
belong to $\cK_j$, the entire mixed slice vanishes.
\end{enumerate}
The third rule is iterated until no further coefficient is deleted.
The process terminates after at most $64$ deletions and depends only on
the three masks.
\end{lemma}

\begin{proof}
The first rule is the vanishing of the three-distinct-block component
of $\J_{\cV}(\cK)$. The second is the normalization of a
rank-one pure block.  For the third, the repeated-block component is a
matrix in $\cK_j$.  A nonzero rank-one symmetric matrix whose image
contains a selector-dense line would force the square of that line to
belong to $\cK_j$, contrary to the selector equations.  Local
independence of the core vectors makes the vanishing coefficientwise.
Since coefficients are only deleted, the iteration is finite.
\end{proof}

Assume that every thin pure block has CP rank one. Under simultaneous
mode and block permutations, the $7^3=343$ mask triples form $23$
orbits. Iterating the three closure rules gives the following orbit
classification:
\begin{equation}\label{eq:triple-mask-partition}
\mathcal O\eqdef\mathcal O_P\sqcup\mathcal O_D\sqcup
\mathcal O_S\sqcup\mathcal O_T,\ \ 
(|\mathcal O_P|,|\mathcal O_D|,|\mathcal O_S|,|\mathcal O_T|)
=(12,3,5,3).
\end{equation}
\begin{center}
\small
\begin{tabular}{@{}p{.08\textwidth}p{.57\textwidth}p{.28\textwidth}@{}}
\toprule
Class & Representatives & Surviving core / condition\\
\midrule
$\mathcal O_P$ & $(1,2,4)$, $(1,2,5)$, $(1,2,7)$, $(1,3,6)$,
$(1,6,6)$, $(1,6,7)$, $(3,5,6)$
& $(1,1,1)$, $(2,2,2)$, $(3,3,3)$\\
& $(3,3,5)$, $(3,3,7)$, $(3,5,7)$, $(3,7,7)$, $(7,7,7)$
& $(1,1,1)$, $(2,2,2)$, $(3,3,3)$\\
$\mathcal O_D$ & $(1,3,5)$, $(1,3,7)$, $(1,7,7)$
& $(1,1,1)$, $(2,2,2)$, $(3,3,3)$, $(4,1,1)$\\
$\mathcal O_S$ & $(1,1,1)$, $(1,1,2)$, $(1,1,3)$, $(1,1,6)$, $(1,1,7)$
& Two split modes, same thick block\\
$\mathcal O_T$ & $(1,2,3)$, $(1,3,3)$, $(3,3,3)$
& A fully thin block\\
\bottomrule
\end{tabular}
\end{center}
For reproducibility, apply rules (i)--(ii) first and then apply all
instances of rule (iii) simultaneously until stable. The successive
deletion counts in the first three rows, in the displayed order, are
\[
\begin{aligned}
&(42),\ (35,2),\ (32),\ (30,2),\ (30,2),\ (28),\ (27);\\
&(29,2),\ (30,2),\ (27),\ (27),\ (27);\\
&(32,4),\ (30,2),\ (28).
\end{aligned}
\]
These are deletions of symbolic positions, independent of coefficient
values. In $\mathcal O_P$ all mode dimensions are at most three;
in $\mathcal O_D$ the $B,C$ dimensions are at most three.
Thus all fifteen orbits contradict mode type $(4,4,4)$.
The last two classes follow directly from the masks and account for
the remaining eight orbits.

\begin{lemma}[A fully thin block with two four-dimensional modes]
\label{lem:triple-full-thin}
Let $\tens P$ be the original five-term residual of a minimal
$167$-term symmetric decomposition. Suppose
\[
(d_A,d_B,d_C)=(4,4,d),\  d\in\{4,5\},\ \ 
d^{\mathrm{proj}}_{A,p}=d^{\mathrm{proj}}_{B,p}=d^{\mathrm{proj}}_{C,p}=1.
\]
Then $d=4$ and $r(\tens P)=4$. After permuting the modes, we have
\begin{equation}\label{eq:triple-all-thin-rho}
(\rho_A,\rho_B,\rho_C)=(4,4,5),\ \ 
\cW_{AB}=\cE_{AB}.
\end{equation}
The five matrices $\va_i\otimes\vb_i$ are nonzero, span a
four-dimensional space, and have a unique relation, supported on
two proportional matrices. Block $p$ is $AC$-good and $BC$-good.
\end{lemma}
\begin{proof}
Write $\tens P_{ppp}=\va\otimes\vb\otimes\vc$.
For each complementary pair $XY,Z$, the source identity gives
\[
p\text{ is }XY\text{-good}
\ \Longrightarrow\ \cW_{XY}\ne\cE_{XY}
\ \Longrightarrow\ \rho_Z=4.
\]
If $p$ were $AB$-good, $AC$-good, and $BC$-good, every residual factor would lie in its
mode space. Since the block is fully thin, we have
\[
(\va_i^{[p]}\otimes\vb_i^{[p]})\in\FF(\va\otimes\vb),
\ \ 
\cW_{AB}=\cE_{AB}+\Span\{\va_i\otimes\vb_i:i\in\set J\}
\]
would put every local source in the local contraction image on a block
that is $AB$-good, contrary to
\cref{lem:triple-good}; the same holds for the other pairs.
The factors cannot all be proportional either. Thus exactly one pair
is proportional, say $\va\parallel\vb$. When $d=5$, saturation
of $AB$ already forces this choice. Consequently
$\rho_A=\rho_B=4$, while $\rho_C=4$ would give the same
all-in-support contradiction. Hence $\rho_C=5$.
If $d=5$, all factors would again lie in their mode spaces,
contradicting that $p$ is $AC$-good. Therefore $d=4$.

Put $\matr M_i\eqdef\va_i\otimes\vb_i$. The $\vc_i$ are independent, so
$\dim\Span\{\matr M_i:1\le i\le5\}=4.$ 
If $\matr M_i=0$, it is the unique zero matrix and the other four
are independent. Their $C$ factors then belong to $\set U_C$.
When $\va_i=0$ (respectively $\vb_i=0$), all five local $AC$
(respectively $BC$) matrices stay in the local contraction image.
Since $p$ is $AC$-good and $BC$-good, the source-generation identity
and \cref{lem:triple-good} exclude
both possibilities. Thus all $\matr M_i$ are nonzero.

For any support $\set I$ of size $m$ and nonzero coefficients $\lambda_i$, put
$\matr A_{\set I}\eqdef[\va_i]_{i\in\set I}$ and
$\matr B_{\set I}\eqdef[\vb_i]_{i\in\set I}$. Then
\cref{lem:factor-calculus}, with $n=5$ and $b=c=4$, gives
\begin{equation}\label{eq:triple-five-source-rank}
\rank\!\left(\sum_{i\in\set I}\lambda_i\va_i\otimes\vb_i\right)
\ge \rank\matr A_{\set I}+\rank\matr B_{\set I}-m\ge m-2.
\end{equation}
Applied to the unique relation, this yields $m\le2$; nonzero
matrices give $m=2$. Combining this proportional pair produces four
terms. Since $d_A=4$, their number is minimal.
\end{proof}

\subsection{Exclusion of three-dimensional modes and residual CP rank at most four}

\begin{lemma}[Three-coordinate rank-one completions]
\label{lem:rank-one-completions}
Let $\matr\Omega\eqdef(\omega_{jk})_{j\ne k}$ be a nonzero off-diagonal
array on $\{1,2,3\}$ admitting a rank-one completion. If its
nonzero entries involve all three indices, all rank-one completions
have a common left-factor line or a common right-factor line.
If they involve only $p,q$, every completion is supported on
$\{p,q\}\times\{p,q\}$.
\end{lemma}
\begin{proof}
Write a completion as $\vx\vy^\top$ and put
$\set I_x\eqdef\operatorname{supp}\vx$ and
$\set I_y\eqdef\operatorname{supp}\vy$.
Its visible edges are precisely
$\{(j,k)\in\set I_x\times\set I_y:j\ne k\}$.
If $\set I_x=\{p\}$ and all three indices occur, both edges from $p$
to the other indices are nonzero. These two columns force every
other row of any completion to vanish; hence the common left line
is $\FF\vect e_p$. The case $|\set I_y|=1$ is transposed.

Suppose $|\set I_x|,|\set I_y|\ge2$. Every row and column in
$\set I_x,\set I_y$ is visible,
and any additional nonzero row or column would create a new edge.
Thus every completion has these same supports. The bipartite graph
with edges $j\ne k$ is connected unless $\set I_x=\set I_y$ is a two-element
set: for unequal two-element sets it is a three-edge path, and it is
also connected for sizes $2\times3$ and $3\times2$ as well as for
size $3\times3$. The exceptional disconnected case involves only two
indices. On a connected graph, the equations
$(\vx)_j(\vy)_k=\omega_{jk}\ne0$ determine the factors up to reciprocal
scaling, so the completed matrix is unique.

Finally, if the visible indices are only $p,q$, choose a nonzero
edge $p\to q$, reversing the names if needed. For the remaining
index $r$, the zero entries $\omega_{rq}=\omega_{pr}=0$ force
$(\vx)_r=(\vy)_r=0$. This proves the last assertion, including the
one-way-edge case.
\end{proof}

\begin{lemma}[Local source test]\label{lem:local-source-test}
For a selected pair $XY$, if block $j$ is $XY$-good, then some residual $XY$ source
matrix has its $jj$ block outside the local contraction image.
\end{lemma}
\begin{proof}
If all residual $jj$ matrices lay in that image, the source-generation
identity \eqref{eq:H345-source-span} would put every original $jj$
matrix there, contradicting \cref{lem:triple-good}.
\end{proof}

\begin{proposition}[Low-dimensional five-term exclusion]
\label{prop:triple-H1}
Suppose that a minimal $167$-term symmetric decomposition exists and let
$\tens P$ be its original five-term residual. Then
\[
r(\tens P)=5,\ \  d_A,d_B,d_C\in\{4,5\}.
\]
\end{proposition}

\begin{proof}
Each mode dimension lies in $\{3,4,5\}$. We first exclude
three-dimensional modes and then CP rank at most four. 
By \eqref{eq:triple-two-three-modes}, two three-dimensional modes
force type $(3,3,3)$. Thus it suffices to consider the four types in
$\{(3,3,3),(3,4,5),(3,4,4),(3,5,5)\}$ and, finally, CP-rank-four
type $(4,4,4)$.

\emph{Type $(3,3,3)$.} Write
\[
\tens P=\sum_{j=1}^3
\va_j\otimes\vb_j\otimes\vc_j,
\ \  \va_j,\vb_j,\vc_j\in\cV_j .
\]
For every pair $XY$, its contraction image is the direct sum of three
nineteen-dimensional local images, and hence has dimension $57$.
The source-generation identity \eqref{eq:H345-source-span} applies.

Let $g_{XY}$ count the $XY$-good blocks. Then
\[
g_{AB}+g_{AC}+g_{BC}\ge6,\ \ 
\dim\cW_{XY}\le58\ \Longrightarrow\ g_{XY}\le1,
\]
where the implication is \cref{lem:triple-source-counting}(ii),
with $e=57$ and at most $59$ sources. Hence at least two source
spaces have dimension $59$. If all three did,
\eqref{eq:H345-source-span} would give $\rho_A=\rho_B=\rho_C=3$;
all residual factors would lie in their mode spaces, making
every block $XY$-bad for every pair $XY$, a contradiction.
Thus exactly two source spaces have dimension $59$; after relabelling,
they are $\cW_{AB}$ and $\cW_{AC}$.
Their sources are bases, and $\rho_B=\rho_C=3$ forces every block
to be $BC$-bad. After rescaling, we have
\[
\tens P=\sum_{j=1}^3
\va_j\otimes\vd_j\otimes\vd_j,
\ \  \va_j\not\parallel\vd_j .
\]
Write the five original residual terms indexed by $i\in\set J$ as
$\vx_i\otimes\vy_i\otimes\vz_i$.
Put $\set D\eqdef\Span\{\vd_1,\vd_2,\vd_3\}$. Since
$\rho_B=\rho_C=3$, we have $\vy_i,\vz_i\in\set D$.
The $AB$ and $AC$ source bases make all these factors nonzero
and make both lists
$\{\vx_i\otimes\vy_i\}_{i\in\set J}$ and
$\{\vx_i\otimes\vz_i\}_{i\in\set J}$ independent.
All three blocks are $AB$-good and $AC$-good.

Use the basis $\vd_1,\vd_2,\vd_3$ to identify
$\vy_i\otimes\vz_i$ with a $3\times3$ matrix.
The off-diagonal image of $\cW_{BC}$ has dimension at most one.
If it is zero, each residual matrix is supported on one diagonal
block. Partition $\set J$ accordingly into $\set J_1,\set J_2,\set J_3$.
Each group is nonempty. If $\set J_j=\{i\}$, comparison of the
complete $A$ coefficient gives $\vx_i\parallel\va_j$.
Every residual $AB$ matrix then has its $jj$ block in
$\cK_j+\FF(\va_j\otimes\vd_j)$: all other groups have zero
$B$ projection there. This contradicts \cref{lem:local-source-test}.
Thus $|\set J_j|\ge2$ for every $j$, impossible with five terms. 
Otherwise choose its nonzero off-diagonal generator $\matr\Omega$ and,
writing $\operatorname{off}$ for deletion of the three diagonal entries, set
\[
\operatorname{off}(\vy_i\otimes\vz_i)=t_i\matr\Omega,
\ \  \set I\eqdef\{i\in\set J:t_i\ne0\}.
\]
The source-generation identity makes $\set I$ nonempty. Since $\tens P$
has no off-diagonal $BC$ component, coefficient comparison gives
$\sum_{i\in\set I}t_i\vx_i=0.$
If $\matr\Omega$ involves all three blocks, \cref{lem:rank-one-completions},
applied to $(\vy_i\otimes\vz_i)/t_i$, gives a common $B$-factor
line or a common $C$-factor line. In the first case write
$\vy_i=\beta_i\vy$ with $\beta_i\ne0$; then
$\sum_{i\in\set I}\frac{t_i}{\beta_i}\vx_i\otimes\vy_i=0,$
contrary to the $AB$ source basis. The second case contradicts
the $AC$ source basis in the same way.

If $\matr\Omega$ involves only $p,q$, \cref{lem:rank-one-completions} puts every
exceptional matrix in those two blocks; a normal residual matrix is
pure diagonal. Thus $\set J=\set J_{pq}\sqcup\set J_r$, where
the $BC$ matrices of the second group are pure in the remaining block.
The preceding single-term argument gives $|\set J_r|\ge2$.
The displayed nontrivial relation is supported on $\set J_{pq}$, so
\[
\dim\Span\{\vx_i:i\in\set J_{pq}\}
\le |\set J_{pq}|-1\le2.
\]
The corresponding terms sum to
$\va_p\otimes\vd_p^2+\va_q\otimes\vd_q^2$, whose
$A$ support is $\Span\{\va_p,\va_q\}$.
Consequently the displayed span equals this two-plane.
Every residual $pp$ matrix for $AB$ therefore lies on
$\FF(\va_p\otimes\vd_p)$, including the zero matrices from
$\set J_r$. This again contradicts \cref{lem:local-source-test}.
This excludes type $(3,3,3)$.

\emph{Type $(3,4,5)$.} Take $d_C=3$ and $d_B=4$. Block-diagonal slices
split the $A,B$ supports. On each thin $B$ line, a quadratic tail
vanishes and the $A$ support is also a line. Thus, up to permutation, we have
\[
(\dim\set U_{A,j},\dim\set U_{B,j})_{j=p,q,r}
=((3,2),(1,1),(1,1)).
\]
The dense $C$ lines give $\kappa_{BC}=\kappa_{AC}=0$, so
$\dim\cE_{BC}=59$ and $\dim\cE_{AC}=58$.
The $BC$ sources form a basis. On $q,r$, \cref{lem:triple-good}
makes the $B,C$ factors proportional; the $A,C$ factors are
therefore independent. Thus $q,r$ are $AC$-good, and their rows have no off-diagonal
output. By \cref{lem:triple-source-counting}(i), they need at least
$58+2=60$ sources, but only $59$ are available.

\emph{Type $(3,4,4)$.} Take $d_C=3$ and write its three pure lines as
$\FF\vc_j$. By \cref{lem:H345-thin}, all $AB$ slices are
block diagonal, so both $\set U_A$ and $\set U_B$ have
pure-intersection dimensions $(2,1,1)$.
Their thick block is the same. Indeed, on a thin $B$ block every
quadratic tail has row space on a selector-dense line and hence
vanishes. The pure block then has rank one, and its $A$ support
is also a line. Interchanging $A,B$ proves the assertion.
Denote the common thick block by $p$ and the thin blocks by $q,r$.

Modulo $\cK$, the $AB$ slices have exactly three independent
directions, one on each physical block. The quadratic tails vanish
in this quotient, and each pure-block direction is nonzero because
$0\notin\cA_j$. Thus $\dim\cE_{AB}=57$.
For a rotated pair, a slice in $\cK$ has $jj$ block
$\vu_j\otimes\vc_j$; symmetry and $\vc_j^2\notin\cK_j$
make it zero. Hence both rotated quadratic intersections vanish and
\begin{equation}\label{eq:H1-344-images}
(\dim\cE_{AB},\dim\cE_{AC},\dim\cE_{BC})=(57,58,58).
\end{equation}
The $q,r$ rows in $AC$ and $BC$ have no off-diagonal output:
only the common thick block can contribute a quadratic tail.

Each thin rank-one pure block is $AC$-good or $BC$-good.
By \cref{lem:triple-source-counting}(i), they cannot both be $AC$-good or both
be $BC$-good. After interchanging $q,r$ if necessary, block $q$ is
$AC$-good and $BC$-bad, while block $r$ is $BC$-good and $AC$-bad.
In particular, $q,r$ are $AB$-good.

If $\dim\cW_{AB}\le58$, then $\cE_{AB}$ is block diagonal and has
codimension at most one in $\cW_{AB}$. The two $AB$-good blocks then require
at least $57+3=60$ sources by \cref{lem:triple-source-counting}(ii),
but only $59$ are available. Hence $\dim\cW_{AB}=59$.
Since $q$ is $AC$-good and $r$ is $BC$-good,
\cref{lem:triple-good} also gives
$\dim\cW_{AC}=\dim\cW_{BC}=59$.
Applying \eqref{eq:H345-source-span} to all three pairs yields
\begin{equation}\label{eq:H1-344-source-dimensions}
\begin{aligned}
2&=\dim(\cW_{AB}/\cE_{AB})\le5-\rho_C,\\
1&=\dim(\cW_{AC}/\cE_{AC})\le5-\rho_B,\\
1&=\dim(\cW_{BC}/\cE_{BC})\le5-\rho_A,
\end{aligned}
\ \  (\rho_A,\rho_B,\rho_C)=(4,4,3).
\end{equation}
Thus every original residual factor lies in its mode space.
Choose the $AC$-good thin block. Both block projections there
are lines, so all five residual $AC$ matrices have their local
block on the pure representative's matrix line. This contradicts
\cref{lem:local-source-test} and excludes type $(3,4,4)$.

\emph{Type $(3,5,5)$.} Take $(d_A,d_B,d_C)=(5,5,3)$.
By \cref{lem:triple-saturation}, the pairs $AC$ and $BC$ are saturated, so
\cref{lem:triple-saturated-pairs} excludes every rank-one pure block.
The three pure $C$ lines give block-diagonal $AB$ slices by
\cref{lem:H345-thin}. Choose a functional nonzero on all five
$C$ factors. Since the $A,B$ factor lists are bases,
\cref{lem:factor-calculus} gives a rank-five block-diagonal slice
whose column space is $\set U_A$. Hence $\set U_A$ splits into
three nonzero pure intersections of total dimension five; at least
one is a line. On that block both the $A$ and $C$ projections are
lines, so its nonzero pure tensor has rank one, a contradiction.

\emph{Rank-four type $(4,4,4)$.} Once three-dimensional modes are excluded,
it follows that $r(\tens P)\le4$ implies $r(\tens P)=4$ and
$\mrank(\tens P)=(4,4,4)$.

By \cref{lem:triple-affine-four}(i), a thin pure block has CP rank
one or two; in the latter case its quadratic tails vanish and the
other two pure blocks have CP rank one.

For a thin $A$ block $q$, write $\matr M$ for its pure contraction
and $j,\ell$ for the other blocks. If $\rank\matr M=2$, there are
no tails. The two active terms have their $B,C$ factors in block $q$; the other two
separate into rank-one representatives on blocks $j,\ell$.  For the
two mode pairs through $A$, the contraction images have dimension $58$ and
the two side rows have no off-diagonal output.  For the opposite pair
the contraction image is block diagonal of dimension $57$, while
$\dim(\cW_{BC}/\cE_{BC})\le1$.  The two side representatives have a common good
pair $XY$, so the two side blocks are $XY$-good.  If $XY$ contains $A$, the $58/59$ two-row argument used in the
$(3,4,4)$ case gives a contradiction; on the opposite pair,
\cref{lem:triple-source-counting}(ii) applies with at most $59=57+2$
sources. Hence every thin pure block has CP rank one, and the three rules
of \cref{lem:triple-core-closure} apply.

In \eqref{eq:triple-mask-partition}, the fifteen orbits of
$\mathcal O_P\cup\mathcal O_D$ have a mode dimension at most three.
It remains to exclude $\mathcal O_S\cup\mathcal O_T$.

For the five remaining orbits in
$\{(1,1,2),(1,1,6),(1,1,3),(1,1,7),(1,2,3)\}$, write
$e_{XY}\eqdef\dim\cE_{XY}$ and
$\kappa\eqdef(\kappa_{AB},\kappa_{AC},\kappa_{BC})$.
An arrow $j\to k$ denotes a possibly nonzero off-diagonal block
with row block $j$ and column block $k$ in the displayed ordering.
Every diagonal quadratic correction disappears modulo $\cK_j$,
so on a rank-one pure block its local image is exactly
$\cK_j+\FF(\vu_{X,j}\otimes\vu_{Y,j})$.

\begin{center}
\small
\begin{tabular}{@{}ccccc@{}}
\toprule
Masks & $\kappa$ & $(e_{AB},e_{AC},e_{BC})$
& Possible arrows $(AB;AC;BC)$ & Side blocks\\
\midrule
$(1,1,2)$ & $(1,0,0)$ & $(57,58,58)$
& $(-;\ 1\to2,3;\ 1\to2,3)$ & $2,3$\\
$(1,1,6)$ & $(1,0,0)$ & $(57,58,58)$
& $(-;\ 1\to2,3;\ 1\to2,3)$ & $2,3$\\
$(1,1,3)$ & $(1,0,0)$ & $(57,58,58)$
& $(-;\ 1\to2,3;\ 1\to2,3)$ & $2,3$\\
$(1,1,7)$ & $(1,0,0)$ & $(57,58,58)$
& $(-;\ 1\to2,3;\ 1\to2,3)$ & $2,3$\\
$(1,2,3)$ & $(0,0,0)$ & $(58,58,58)$
& $(1\to2;\ 1\to2;\ 2\to1)$ & $1,3$\\
\bottomrule
\end{tabular}
\end{center}
Here $-$ means that all off-diagonal blocks vanish; an allowed
arrow need not occur.

For the first four rows, both $A$ and $B$ split with the same thick block
$p=1$, while $\set U_{C,p}$ is a selector-dense line. Consider a cross
slot $A_pB_q$ with $q\ne p$, keeping its full $C$ coefficient.
Its $C_q$ component vanishes because the repeated $B_q$ line is
selector-dense; its third-block component vanishes by the affine
block equations. The coefficient therefore belongs to $\set U_{C,p}$.
The remaining $A_pC_p$ quadratic factor contains this dense line,
so it vanishes as well. Transposition treats $A_qB_p$, and a
cross slot between the two thin blocks vanishes by the same
repeated-line argument. Thus every $AB$ slice is block diagonal,
and all mixed components have the form
\[
\tens H\in\cK_p\otimes(\cV_q\oplus\cV_r)
\quad\text{in the $AB\mid C$ ordering}.
\]
The two side blocks have rank one. After reducing $AB$ modulo
$\cK$, the $p$ coefficient lies on the line $\set U_{C,p}$, while
the $q,r$ coefficients lie on their rank-one pure blocks' $C$-factor
lines. Each local class is nonzero, so the quotient slice image
has dimension three.
Since $d_C=4$, this gives $\kappa_{AB}=1$ and $e_{AB}=57$.
For $(1,1,2)$ the thick blocks of $A,C$ and of $B,C$ differ;
for the other three rows $C$ is nonsplit. In either case
\cref{lem:H345-four-support} gives
$\kappa_{AC}=\kappa_{BC}=0$. The displayed form of $\tens H$
gives exactly the permitted arrows in the table.

For some $XY$, blocks $q,r$ are $XY$-good. If $XY=AB$,
\eqref{eq:H345-source-span} gives
$\dim(\cW_{AB}/\cE_{AB})\le5-d_C=1$;
\cref{lem:triple-source-counting}(ii) would require $60$ sources.
If $XY\ne AB$, the two $XY$-good side rows have no
off-diagonal output and \cref{lem:triple-source-counting}(i) again requires
$58+2=60$ sources. Thus all four rows are excluded.

For $(1,2,3)$, the mode $A$ splits with thick block $1$ while
$B$ splits with thick block $2$, and $C$ is nonsplit. Thus
\[
\cK\cap(\set U_X\otimes\set U_Y)=0
\ \ (XY=AB,AC,BC)
\]
by \cref{lem:H345-four-support}. In particular all three
$\kappa_{XY}$ vanish, proving the last row's dimensions.
Block $3$ is thin in every mode. By \cref{lem:H345-thin}, each
tail in a contraction on that block lies in one of the displayed
zero intersections. These tails vanish; the repeated dense-line
and three-distinct-block equations kill the other mixed slots
containing $3$. Hence block $3$ is isolated.

Among blocks $1,2$, the thin projections $A_2,B_1$ give
\[
\tens P_{21k}=\tens P_{2k1}=\tens P_{k12}=0
\ \ (k=1,2,3).
\]
For $k=1,2$, every indicated mixed slot repeats a block in a pair
containing one of these dense thin lines; $k=3$ is already excluded.
Thus the only possible arrows are $1\to2$ for $AB,AC$ and
$2\to1$ for $BC$, as claimed. All three pure blocks have CP rank
one, since each has a thin mode.

Choose a pair $XY$ for which blocks $1,3$ are $XY$-good. For $XY=AB$ or $AC$,
transpose the matrix ordering; their only possible arrow becomes
$2\to1$, so rows $1,3$ have no off-diagonal output.
For $XY=BC$ this is already the displayed ordering. In every case
the contraction image has dimension $58$, and
\cref{lem:triple-source-counting}(i) contradicts the $59$-source budget.
This completes the five orbit exclusions.
For $(1,1,1)$, two blocks are fully thin. The source identity and
\cref{lem:triple-full-thin} give a unique relation supported on two
proportional nonzero residual matrices. Applying the good rotated
mode pair of each rank-one representative forces their repeated matrix to have two
nonzero diagonal blocks and a zero cross block, impossible in rank
one. Thus only $(1,3,3)$ and $(3,3,3)$ remain.

In either remaining orbit block $3$ is thin in all three modes. Every
mode pair contains a nonsplit four-dimensional mode space, so its residual
slice space has zero intersection with $\cK$. Contracting a thin
line, any possible tail would have its row or column space on a
selector-dense pure line and is therefore zero.  The third block
really separates:
\[
\tens P=\tens P_{12}\oplus
(\va\otimes\vb\otimes\vc),
\ \  r(\tens P_{12})=3,\ 
\mrank(\tens P_{12})=(3,3,3).
\]

Apply \cref{lem:triple-full-thin} to the five original residual
terms. Their span dimensions are $(4,4,5)$ after permuting the modes.
The unique relation among the five nonzero $AB$ matrices is supported
on a proportional pair.
Write $\matr D$ for its common direction.
The isolated block $3$ is $AC$-good and
$BC$-good.  The matrix $\matr D$ together with the other three residual
$AB$ matrices is a basis of the $AB$ slice space.  Hence the three
nonrepeated $C$ coefficients already belong to the mode space
$\set U_C$.  After combining the pair, the resulting four-term
CP decomposition is minimal, so its $A$ and $B$ factor lists are bases.
Contraction by their dual bases isolates each nonrepeated
$\va_i\otimes\vc_i$ and $\vb_i\otimes\vc_i$ as a slice
direction.
Thus only the two individual members of the repeated pair can supply
a local source matrix outside the corresponding local contraction image
in $AC$ or $BC$. The $AC$-good
third block therefore forces the $A$ factor of $\matr D$ to
have a nonzero third-block component, and the $BC$ escape forces the
same for its $B$ factor.  On the other hand, the equality $\rho_C=5$ gives
$\cW_{AB}=\cE_{AB}$, which is block
diagonal for the decomposition into the first two copies and the
third copy.  A rank-one matrix in this space with nonzero $33$ block
is pure in that block.  Thus
\[
\matr D\in\FF(\va\otimes\vb)
\ \text{and is pure in block }3.
\] 
For this decomposition, the $AB$ slice space is the span of four
coordinate rank-one matrices and has no other rank-one directions.
Since the isolated matrix
$\va\otimes\vb$ belongs to this slice space, it is exactly the
coordinate matrix obtained by combining the repeated pair.  The other
three original terms therefore sum to $\tens P_{12}$, while the
repeated pair gives the isolated rank-one block.

Put
$\set H_{12}\eqdef(\cV_1\oplus\cZ_1)
\oplus(\cV_2\oplus\cZ_2), 
\set H_3\eqdef\cV_3\oplus\cZ_3$ and 
$\set H\eqdef\set H_{12}\oplus\set H_3.$ 
We first note a complete-coefficient fact. Put
$\cV_{12}\eqdef\cV_1\oplus\cV_2$ and let $\matr p_g$
project onto $\cV_g$ for $g\in\{12,3\}$. For every selected pair $X,Y$, we have
\[
(\matr p_g\otimes\matr p_g\otimes\id)\tens G_{XY}
\in(\cV_g\otimes\cV_g)\otimes\set H_g,
\ \  g\in\{12,3\}.
\]
Indeed, a term from another complete group has two base factors
in that group. At least one is in a retained position, is fixed by
the corresponding substitution, and is killed by $\matr p_g$.
Thus the assertion does not require block-preserving substitution maps.

The $AB$ contraction image is block diagonal of dimensions
$39+19=58$.  Its $59$ original matrices have exactly the one
repeated-pair relation described above.  Since a rank-one matrix in a
block-diagonal two-group space is pure in one group, there are $39$
independent matrix sources in the first group and $20$ in the third;
the sole relation is the repeated pair in the third group.  Comparing
the complete coefficients puts each nonrepeated coefficient vector in the
corresponding $\set H$-group and puts the related combination of the
repeated pair in $\set H_3$.

The third block is $AC$-good.  Hence $\cW_{AC}\ne\cE_{AC}$;
the quotient bound makes $\dim\cW_{AC}=59$, so its $59$ original
matrices form a basis and
$\cW_{AC}=\cE_{AC}\oplus\FF\matr\Omega.$ 
The source-generation identity allows $\matr\Omega$ to be chosen as
the $AC$ matrix of one member of the repeated residual pair.  Its
$A$ factor is pure in the third group, so its only possible
cross-group block is $3\to12$; moreover its $33$ block lies outside
the third-block local contraction image, since block $3$ is $AC$-good and
\cref{lem:local-source-test} applies. If the
$3\to12$ block were nonzero, every normalized exceptional matrix
would have that fixed block and a nonzero escaping $33$ block.  The
vanishing reverse block then makes all its complete $A$ factors
collinear, say
$\matr\Pi_A\vh_i=\alpha_i\va$.  The quotient coefficient relation
$\sum_{t_i\ne0}t_i\vh_i=0$
would give
\[
\sum_{t_i\ne0}\frac{t_i}{\alpha_i}
(\matr\Pi_A\vh_i\otimes\matr\Pi_C\vh_i)=0,
\]
contradicting the basis property.  Thus $\matr\Omega$, and therefore
all of $\cW_{AC}$, is block diagonal.  Its two group dimensions are
$39$ and $20$, so complete-coefficient comparison places every one of
its coefficient vectors individually in the corresponding complete group.
In particular, the two repeated residual vectors are separately in
$\set H_3$.  The same argument for $BC$ treats the remaining original
source-matrix set.

Consequently all $167$ original decomposition vectors split into
\[
36+36+36+3=111
\ \text{vectors in }\set H_{12},
\ \ 
18+18+18+2=56
\ \text{vectors in }\set H_3.
\]
Projection to $\set H_{12}$ gives a $111$-term symmetric decomposition of
$\tens{V}_2$, contradicting the already proved equality
$s(\tens{V}_2)=112$.  This excludes CP rank at most four and completes
the proof.
\end{proof}

\subsection{Exclusion of CP-rank-five residuals of multilinear rank \texorpdfstring{$(4,4,4)$}{(4,4,4)}}

\begin{lemma}[The preliminary \((4,4,4)\) source exclusions]
\label{lem:triple-444-preliminary}
Let $\tens P$ be the original five-term residual of a fixed minimal
$167$-term symmetric decomposition, and suppose
$r(\tens P)=5$ and $\mrank(\tens P)=(4,4,4)$.
\begin{enumerate}
\item Two mode spaces cannot both have pure-intersection dimensions
$(2,1,1)$ with the same
two-dimensional physical block.
\item There is no physical block on which all three mode projections
are one-dimensional.
\end{enumerate}
\end{lemma}

\begin{proof}
The second assertion follows from \cref{lem:triple-full-thin}:
a fully thin block would give $r(\tens P)=4$.

For the first assertion, let $p$ be the common thick block and let
$q,r$ be the two thin blocks.  The mixed-slot equations eliminate all
cross blocks except possibly a $\cK_p$-matrix with free coefficient
in $\cV_q\oplus\cV_r$.  The pure blocks at $q,r$ are rank-one affine
representatives.  More precisely, let $C$ be the third mode.  It
cannot be thin at $q$ or $r$, by the second assertion just proved;
hence it is nonsplit.  Indeed, a four-dimensional mode is generated
by one pure line in each physical block and one additional vector;
if it split, that additional vector would lie in one block and the
other two block projections would be one-dimensional.  Thus each
$\set U_C\cap\cV_j$ is a selector-dense line.  Consider an
$A_pB_q$ coefficient and retain its complete $C$ vector.  Its
$C_q$ component is killed by the dense $B_q$ line and the
$BC$--$\cK_q$ condition, its $C_r$ component is a
three-distinct-block coefficient, and its remaining $C_p$ component
is killed by the dense line $\set U_C\cap\cV_p$ and the
$AC$--$\cK_p$ condition.  The slots $A_qB_p$ and $A_qB_r$ are
identical after permuting the roles.  Thus every $AB$ cross block
vanishes and the only nonpure part has the form
$\tens H\in
\cK_p\otimes(\cV_q\oplus\cV_r)$ in the $AB\mid C$ ordering.
Since $\set U_C$ is nonsplit, \cref{lem:H345-four-support} gives
zero intersection with $\cK$ for both the $AC$ and $BC$ slice
spaces. Their contraction images therefore have dimension $58$, and the
thin rows $q,r$ have no off-diagonal output.
The $AB$ image is block diagonal.  By
\cref{lem:H345-four-support}, its residual slice space meets
$\cK$ in dimension at most one; hence the image has dimension $57$
when that intersection is nonzero and dimension $58$ otherwise.

The two rank-one blocks $q,r$ have a common good mode pair $XY$, so
both blocks are $XY$-good and the corresponding source span satisfies
$\dim(\cW_{XY}/\cE_{XY})\le1$.  If $XY=AB$,
then $\cE_{AB}$ is block diagonal, has dimension $57$ or $58$, and is
supplied by at most $59$ nonzero original sources. Item~(ii) of
\cref{lem:triple-source-counting} gives an immediate contradiction. 
If $XY=AC$ or $BC$, the
contraction image has dimension $58$ and the two corresponding block
rows have no off-diagonal output.
\Cref{lem:triple-source-counting}(i) gives the remaining contradiction, proving
the first assertion.
\end{proof}

\begin{proposition}[Exclusion of a CP-rank-five \((4,4,4)\) residual]
\label{prop:triple-H2}
No residual compatible with a minimal $167$-term symmetric decomposition can
satisfy
$r(\tens P)=5$ and 
$\mrank(\tens P)=(4,4,4).$ 
\end{proposition}

\begin{proof}
We first prove the stronger quadratic intersections
\begin{equation}\label{eq:triple-environment-zero}
\cK\cap(\set U_X\otimes\set U_Y)=0
\ \ (XY=AB,AC,BC).
\end{equation}
A nonzero matrix in such an intersection would, by
\cref{lem:H345-four-support} applied to both modes, force them to
have pure-intersection dimensions $(2,1,1)$ with the same thick block.
This contradicts
\cref{lem:triple-444-preliminary}, proving
\eqref{eq:triple-environment-zero}.

Next suppose that block $p$ has one-dimensional projection in the
$A$ mode.  Let $\FF\va$ be that line and choose a functional supported
on $\cV_p$ and equal to one on $\va$.  The contracted $BC$ matrix is
block diagonal:
\[
\matr M_p\oplus\bigoplus_{j\ne p}\matr N_j,
\ \  \matr N_j\in\cK_j.
\]
Its block row and column spaces lie in the corresponding pure
intersections.  Hence every $\matr N_j$ belongs to
$\cK\cap(\set U_B\otimes\set U_C)$ and vanishes by
\eqref{eq:triple-environment-zero}.  If
$\rank\matr M_p\ge2$, both $\set U_B\cap\cV_p$ and
$\set U_C\cap\cV_p$ have dimension two, contradicting the first part
of \cref{lem:triple-444-preliminary}.  Thus
$\rank\matr M_p=1$.  We have proved:
\begin{equation}\label{eq:triple-thin-purification}
\dim\matr p_p\set U_X=1
\ \Longrightarrow\ 
r(\tens P_{ppp})=1.
\end{equation} 
By the orbit classification \eqref{eq:triple-mask-partition}, the union
$\mathcal O_P\cup\mathcal O_D$ contradicts mode type $(4,4,4)$,
while $\mathcal O_S\cup\mathcal O_T$ contradicts
\cref{lem:triple-444-preliminary}. These exhaust all $23$ orbits.
\end{proof}

\subsection{The remaining fixed-source classification}
\label{sec:triple-H345}

Throughout this subsection suppose that a minimal $167$-term symmetric
decomposition has been fixed, together with its original three
substitution maps.  Thus the residual
\begin{equation}\label{eq:H345-residual}
\tens P\eqdef\sum_{i=1}^{5}\va_i\otimes\vb_i\otimes\vc_i
\in\cA^{[3]}
\end{equation}
is compatible with those fixed sources, has CP rank five, and, by
\cref{prop:triple-H1,prop:triple-H2}, has multilinear rank in
$\{(4,4,5),(4,5,5),(5,5,5)\}$.
We retain the block intersections, projection dimensions, and original
source spaces defined above. The following arguments use the preceding
source-space and normalization lemmas.

\begin{lemma}[Single-block peeling with fixed sources]
\label{lem:H345-fixed-peeling}
Suppose the original residual $\tens P$ has at least two saturated
mode pairs.  There do not exist a block $p$, a nonzero
$\matr N\in\cK_p$, and a vector $\vx$ such that
\begin{equation}\label{eq:H345-fixed-peeling}
\tens Q\eqdef\tens P-\vx\otimes\matr N\in\cA^{[3]},\ \ 
r(\tens Q)=4,\ \ \mrank(\tens Q)=(4,4,4),
\end{equation}
after a possible permutation of the modes.
\end{lemma}
\begin{proof}
Suppose such a residual exists. By
\cref{lem:triple-affine-four}, it has a rank-one pure block. 
The two pure blocks outside $p$ are unchanged and have rank at
least two by \cref{lem:triple-saturated-pairs}. Apply
\cref{lem:triple-affine-four}(ii) and then
\cref{lem:H345-R333}.  For either side block, its quotient modulo the
local quadratic space equals the quotient of the rank-one tensor
$\tens R_j$ in \eqref{eq:H345-side-quotient}; the term
$\lambda_j\vw_j^3$ vanishes
in every two-mode quadratic quotient. Thus the unchanged side block
of the original $\tens P$ contradicts
\cref{lem:triple-saturated-pairs}. The tensor $\tens Q$ is used only
as an affine CP-rank-four reduction; the obstruction is applied to the
unchanged residual $\tens P$.
\end{proof}

\begin{corollary}[Single-block quadratic slices in $(5,5,5)$]
\label{cor:single-block-slice}
For an original five-term residual of mode type $(5,5,5)$, we have
\begin{equation}\label{eq:single-block-slice}
\set L_{XY}(\tens P)\cap\cK_p=0
\ \ (XY=AB,AC,BC;\ p=1,2,3).
\end{equation}
\end{corollary}
\begin{proof}
All three factor lists are bases, and every original pair is
saturated by \cref{lem:triple-saturation}. If
$0\ne\matr N\in\set L_{BC}(\tens P)\cap\cK_p$, write
$\matr N=\sum_i n_i\vb_i\otimes\vc_i$ and take $n_1\ne0$.
By \cref{lem:factor-calculus}, the tensor
$\tens Q\eqdef\tens P-(\va_1/n_1)\otimes\matr N$ has rank four
and mode type $(4,4,4)$. Since $\matr N\in\cK_p$, it also belongs to
$\cA^{[3]}$, contrary to \cref{lem:H345-fixed-peeling}.
Permuting the modes proves the other cases.
\end{proof}

\begin{remark*}[Affine reduction versus source-compatible peeling]
The subtraction in \cref{cor:single-block-slice} is the familiar
substitution operation \cite[Proposition~3.11]{buczynski2020additivity}.
Direct-sum-preserving rank-one peeling is also available in
\cite[Proposition~5.2.5]{rupniewski2024additivity}.
Here the quadratic slice need not have rank one, and the obstruction
must still be applied to the unchanged original tensor.
\Cref{lem:H345-fixed-peeling} supplies that additional step; it assigns
no new sources to the four-term tensor.
\end{remark*}

\begin{lemma}[Source-compatible splitting obstruction]
\label{lem:H345-splitting}
Under the standing fixed-source hypotheses, no mode is split in
mode type $(4,5,5)$ or $(5,5,5)$.
\end{lemma}
\begin{proof}
Two modes have dimension five, so their complementary pairs are
saturated by \cref{lem:triple-saturation}. Hence every pure block
has rank at least two by \cref{lem:triple-saturated-pairs}.
If all three modes split, the residual is impossible. In type $(4,5,5)$,
the four-dimensional mode has local
dimensions $(2,1,1)$, so among the three modes at least two are thin on
one common block, making that pure block rank one.  In type $(5,5,5)$,
the only distribution without two thin modes in one block is,
after relabelling, given by
\[
\begin{array}{c|ccc}
&A&B&C\\ \hline
1&1&2&2\\2&2&1&2\\3&2&2&1.
\end{array}
\]
For the thin $A$ line in block $1$, each $\cK$ tail has a
selector-dense thin line on one of its two sides (the $B$ line in
block $2$ and the $C$ line in block $3$), so both tails vanish; the
same holds cyclically.  The pure matrices have rank two, rather than
one, by saturation.  In a five-term CP decomposition, independence of the
five factors in each mode shows that each contraction uses exactly
two indices and that both factors of each active term are pure in
the relevant block.  The active sets for the thin $A,B,C$ lines are
then pairwise disjoint: an index shared by the first two, for
example, would have its common $C$ factor pure in both blocks $1$
and $2$; the other two pairs are identical after cycling the modes.
This is impossible for three two-element subsets of a five-element
set.

Consider $(4,5,5)$, with $d_A=4$.  If $A$ splits, its two thin blocks
and \cref{lem:H345-thin} force both $\set U_B$ and $\set U_C$ to have pure
intersection dimensions at least $2,2,1$.  Hence they split, which
was just excluded.  If a five-dimensional mode splits, it has a
thin block; the same contraction forces the four-dimensional mode
to have local dimensions at least $2,1,1$, so it splits and we are
again in the preceding case.

Now consider $(5,5,5)$.  A split mode of type $(3,1,1)$ has two thin
blocks; thin contraction forces each other mode to have local
dimensions at least $(2,2,1)$, so all three split.  If two modes of
type $(2,2,1)$ split, equal thin blocks give a rank-one pure block,
whereas distinct thin blocks force the third mode to have local
dimensions at least $(2,2,1)$; both alternatives are impossible.
Thus a remaining split possibility consists of exactly one split
mode, say $A$, of type $(2,2,1)$, with thin block $p$. 
Write $\tens P_{ppp}=\va\otimes\matr M$.  By
\cref{lem:H345-thin}, the thin contraction is
$\matr M\oplus\matr N_q\oplus\matr N_r$.  The pure block is not
rank one, so $\rank\matr M\ge2$.  The nonsplit five-dimensional
modes $B,C$ have at most a two-dimensional pure intersection in any
block.  Hence $\rank\matr M=2$, both $p$ intersections are planes,
the other four intersections are selector-dense lines, and the two
$\cK$ tails vanish.  Because the five $B$ factors and five $C$
factors in \eqref{eq:H345-residual} are bases, exactly two indices,
namely $1$ and $2$, occur in this contraction, and we have
\[
\matr Q_i\eqdef\vb_i\otimes\vc_i\in\cV_p\otimes\cV_p,\ \ 
\set D\eqdef\Span\{\matr Q_1,\matr Q_2\}
\subseteq\set L_{BC}(\tens P)\cap(\cK_p+\FF\matr M).
\]
Thus $\set D\cap\cK_p$ is nonzero, contrary to
\cref{cor:single-block-slice}, since $\set D$ consists of slices of
$\tens P$ for the pair $BC$. This excludes the last single-split $(5,5,5)$ case.
\end{proof}

We can now prove the three remaining classification statements.

\begin{lemma}[A common thin-block slot]
\label{lem:H345-common-thin-slot}
Let $\tens P\in\cA^{[3]}$ and suppose that the same block $j$ is
thin in modes $A$ and $B$. Choose nonzero vectors $\va_j,\vb_j$
spanning the corresponding projection lines. Then, for suitable
$\vz_j\in\cV_j$ and $\tens H_j\in\cK\otimes\cV_j$, we have
\begin{equation}\label{eq:H345-common-thin}
((\id\otimes\id)\otimes\matr p_j)\tens P
=\va_j\otimes\vb_j\otimes\vz_j+\tens H_j.
\end{equation}
\end{lemma}
\begin{proof}
This follows directly from the repeated-block and
three-distinct-block slot equations in \eqref{eq:triple-slots}; the
argument does not require either mode to be split.
\end{proof}

\begin{proposition}[The $(4,4,5)$ residual]
\label{prop:triple-445}
For the fixed residual, suppose $d_A=d_B=4$ and $d_C=5$. Then
$\kappa_{AB}=\kappa_{AC}=\kappa_{BC}=0$. Moreover, at least one of
$\set U_A,\set U_B$ projects two-dimensionally onto every physical block.
\end{proposition}
\begin{proof}
Since $d_C=5$, we have $\cW_{AB}=\cE_{AB}$. Moreover, no block is fully
thin by \cref{lem:triple-full-thin}. We first prove $\kappa_{AB}=0$.
Otherwise \cref{lem:H345-four-support} gives
\[
\set L_{AB}(\tens P)\cap\cK=\FF\matr N,\ \ 
\supp\matr N=\{p\},\ \ \rank\matr N\in\{1,2\}.
\]
Both four-dimensional modes split with common thick block $p$.
On the other blocks, saturation forces
$\tens P_{jjj}=\vd_j^2\otimes\vz_j,\ \ 
\vd_j\not\parallel\vz_j\ (j=q,r).$ 
Thus $q,r$ are both $AC$-good and $BC$-good.

\emph{Rank-one intersection.}
After rescaling the chosen generator of the one-dimensional
intersection, write $\matr N=\vw^2$. Since the
five $C$ factors form a basis, the five matrices
$\matr Q_i^{AB}\eqdef\va_i\otimes\vb_i$ form a basis of
$\set L_{AB}(\tens P)$. Let $\set I$ be the support of the unique expansion
\begin{equation}\label{eq:H345-square-expansion}
\vw^2=\sum_{i\in\set I}\lambda_i\matr Q_i^{AB},\ \ \lambda_i\ne0\ (i\in\set I).
\end{equation}
The thin-line and three-distinct-block equations give
$(\matr Q_i^{AB})_{qr}=(\matr Q_i^{AB})_{rq}=0$. In the complete $C_j$
projection, the only nonpure coefficient is a multiple of $\vw^2$.
Write $[\vu]_{\vz_j}\eqdef\vu+\FF\vz_j$; then
\begin{align}
((\id\otimes\id)\otimes\matr p_j)\tens P
&=\vd_j^2\otimes\vz_j+\vw^2\otimes\vx_j,
\ \ j=q,r,\label{eq:H345-square-Cj}\\
[\matr p_j\vc_i]_{\vz_j}&=\lambda_i[\vx_j]_{\vz_j}\ (i\in\set I),\ \
[\matr p_j\vc_i]_{\vz_j}=0\ (i\notin\set I).
\label{eq:H345-square-coefficients}
\end{align}
Since $q,r$ are $AC$-good and $BC$-good, the source-generation identity requires
\begin{equation}\label{eq:H345-square-escape}
\matr p_j\Span\{\va_i:i\in\set I\}\ne0,\ \
\matr p_j\Span\{\vb_i:i\in\set I\}\ne0\ (j=q,r).
\end{equation}
Indeed, failure of the first or second condition puts every residual $AC$
or $BC$ source matrix, respectively, and hence every corresponding original
source matrix, in the local contraction image of an $AC$-good or $BC$-good block.

By \eqref{eq:triple-five-source-rank}, we have
$1\ge |\set I|-2$, so $|\set I|\le3$.
If $|\set I|\le2$, one restricted factor span is the line $\FF\vw$,
contradicting \eqref{eq:H345-square-escape}. For $|\set I|=3$, both
restricted spans are planes. Choose $\vx,\vy\in\cV_q\oplus\cV_r$ and write
\[
\va_i=\alpha_i\vx+\gamma_i\vw,\ \
\vb_i=\beta_i\vy+\delta_i\vw,
\ \ \matr p_q\vx,\matr p_r\vx,\matr p_q\vy,\matr p_r\vy\ne0.
\]
The zero $qr$ block gives $\alpha_i\beta_i=0$. Hence
$\set H_A\eqdef\{i:\alpha_i\ne0\}$ and
$\set H_B\eqdef\{i:\beta_i\ne0\}$ are
disjoint. The $\vx\otimes\vw$ and $\vw\otimes\vy$ coefficients give
$\sum_{i\in\set H_A}\lambda_i\alpha_i\delta_i=0$ and
$\sum_{i\in\set H_B}\lambda_i\gamma_i\beta_i=0.$
Every displayed term is nonzero, so $|\set H_A|,|\set H_B|\ge2$,
impossible inside the three-element set $\set I$.

\emph{Rank-two intersection.}
Put $\set W\eqdef\im\matr N$. Symmetry gives
$\set U_{A,p}=\set U_{B,p}=\set W$ and
\begin{equation}\label{eq:H345-WK-line}
\cK_p\cap\Sym^2\set W=\FF\matr N.
\end{equation}
The $C$ projections on $q,r$ are thick, since a fully thin block is
excluded. Thus $\dim\set U_{C,p}\le2$. The mixed $A_pB_pC_j$
components with $j=q,r$ are multiples of $\matr N$; all other
components with $C$ in $q,r$ are pure, by the thin-line slot equations.
Choose $\phi$ on $\set W\otimes\set W$ with $\phi(\matr N)=1$, extend it
by zero outside block $p$, and set
\[
\vz\eqdef(\phi\otimes\id)\tens P\in\set U_C,
\ \  \tens P'\eqdef\tens P-\matr N\otimes\vz\in\cA^{[3]}.
\]
Its side components use only the pure directions $\vz_q,\vz_r$.
Subtracting those components yields
\begin{equation}\label{eq:H345-Pprime-support}
\tens P'_{ppp}\in
\cA_p\cap\bigl(\set W\otimes\set W\otimes\set U_{C,p}\bigr).
\end{equation}
If $\set W\subseteq\set U_{C,p}$, then
$\set W=\set U_{C,p}$; symmetrization and
the rank-two quadratic $\matr N$ contradict
\cref{lem:binary,rem:symmetrization}. Hence
\begin{equation}\label{eq:H345-W-not-C}
\set W\not\subseteq\set U_{C,p}.
\end{equation}

A tail $A_jB_pC_p$ has a symmetric $BC$ coefficient with column
space in $\set W$, so it is a multiple of $\matr N$. If nonzero, the
thin $A_j$ contraction puts its row space $\set W$ inside $\set U_{C,p}$,
contradicting \eqref{eq:H345-W-not-C}. The transposed argument
kills $A_pB_jC_p$. Thus the $q,r$ rows have no off-diagonal output.
The same support argument gives
$\cK\cap(\set U_A\otimes\set U_C)
=\cK\cap(\set U_B\otimes\set U_C)=0:$ 
their $p$ blocks would be multiples of $\matr N$, while the side
blocks vanish on dense lines. Consequently both rotated images have
dimension $58$, and blocks $q,r$ are both $AC$-good and $BC$-good. By
\cref{lem:triple-source-counting}(i),
this is impossible. We have proved $\kappa_{AB}=0$.

\emph{Nonsplitting.}
No block is thin in both modes $A$ and $B$. Indeed, if $j$ were such
a block, \cref{lem:H345-common-thin-slot} would apply. A $C_j$
functional annihilating $\vz_j$ in
\eqref{eq:H345-common-thin} would give a quadratic $AB$ slice.
Since $\kappa_{AB}=0$, the $C_j$ projection would be a line,
contradicting the fully thin exclusion. If both $A,B$ split, they
would have a common thin block, which is therefore impossible.

Suppose only $A$ splits, with thin blocks $q,r$. By
\cref{lem:H345-thin}, their contractions have no tails, because
the nonsplit $B$ intersections are dense lines. Their pure blocks
have rank one and, by $AB$ saturation, are $AC$-good and $BC$-good.
For $\set S_j\eqdef\{i:\matr p_j\va_i\ne0\}$, independence of the $C$
factors gives $\vb_i\in\set U_{B,j}$ for $i\in\set S_j$.
The $AC$-goodness of block $j$ forces $|\set S_j|\ge2$; with one active index, every local
source matrix would remain in the local contraction image. The sets
$\set S_q,\set S_r$ are disjoint, so
\[
4\le\dim\Span\{\vb_i\}_{i=1}^5
\le 1+1+(5-|\set S_q|-|\set S_r|)\le3,
\]
a contradiction. Interchanging $A,B$ proves both are nonsplit.
Then \cref{lem:H345-four-support} gives
$\kappa_{AC}=\kappa_{BC}=0$.

\emph{Projection masks.}
Put $\set S_X\eqdef\{j:d^{\mathrm{proj}}_{X,j}=2\}$ for $X=A,B$.
Nonsplitting gives $|\set S_A|,|\set S_B|\ge2$, while
the common thin-block exclusion just proved gives
$\set S_A\cup\set S_B=\{1,2,3\}$.
Thus the only mask pairs $(\sigma_A,\sigma_B)$ are
$(3,5),(3,7),(7,7)$. Suppose $(\sigma_A,\sigma_B)=(3,5)$, and let $p$ be thin for $A$.
Its contraction is a rank-one matrix supported on one block
$\vb_p\otimes\vz_p$: the $B$ pure intersections are dense lines.
For $\set I_p\eqdef\{i:\matr p_p\va_i\ne0\}$, independence of the $C$
factors and the $AC$-goodness and $BC$-goodness of block $p$ give
\[
\vb_i\in\FF\vb_p\ (i\in\set I_p),\ \  |\set I_p|\ge2.
\]
The five $B$ factors have rank four, so their relation space has
dimension one. Hence $|\set I_p|=2$ and this pair supports that relation.
Writing $\set I_p=\{1,2\}$, choose $0\ne\vu_p\in\matr p_p\set U_A$
and $s_1,s_2\ne0$ with $\matr p_p\va_i=s_i\vu_p$ $(i=1,2)$, and, after rescaling,
write $\vb_1=\vb_p$ and $\vb_2=\mu\vb_p$. Then
\[
s_1\vc_1+\mu s_2\vc_2=\vz_p\in\cV_p,
\ \ \mu s_1s_2\ne0.
\]
If their $C_p$ projections spanned a line, every residual local
$AC$ source matrix would stay in the local contraction image at block $p$.
Thus the $AC$-goodness of $p$ gives
\begin{equation}\label{eq:H345-duplicate-plane}
\dim\matr p_p\Span\{\vc_i:i\in\set I_p\}=2.
\end{equation}
For the different thin block $q$ of $B$, the symmetric argument gives
a two-element set $\set T_q$, with $\va_i\in\FF\va_q$ for
$i\in\set T_q$ and $\dim\matr p_q\Span\{\vc_i:i\in\set T_q\}=2$.

The $AC$-goodness of $p$ and \eqref{eq:H345-source-span} give
$\cW_{AC}=\cE_{AC}\oplus\FF\matr\Omega$, where
$\matr\Omega_{pp}$ lies outside the local image. Therefore
\[
\matr Q\in\cW_{AC},\ \matr Q_{pp}=0
\ \Longrightarrow\ \matr Q\in\cE_{AC}.
\]
For each $i\in\set T_q$, apply this to $\va_i\otimes\vc_i$.
Its $qq$ block belongs to $\cK_q+\FF(\va_q\otimes\vz_q)$.
Since $\va_i\in\FF\va_q$ and
$\cK_q\cap(\FF\va_q\otimes\cV_q)=0$, it follows that
$\matr p_q\vc_i\in\FF\vz_q$. This contradicts the preceding
two-dimensional projection. Hence $(3,5)$ is excluded.
Thus, after interchanging $A,B$ if necessary,
$\set S_B=\{1,2,3\}$, which proves the projection assertion.
\end{proof}

\begin{proposition}[The $(4,5,5)$ and $(5,5,5)$ residuals]
\label{prop:triple-455}
For the fixed original five-term residual, suppose
$(d_A,d_B,d_C)=(d,5,5)$, where $d\in\{4,5\}$. Then
\begin{equation}\label{eq:H345-H4-conclusion}
\kappa_{AB}=\kappa_{AC}=\kappa_{BC}=0,\ \ 
d^{\mathrm{proj}}_{A,j}\ge2\ (j=1,2,3).
\end{equation}
When $d=5$, cyclic symmetry also gives
\begin{equation}\label{eq:H345-H5-conclusion}
d^{\mathrm{proj}}_{X,j}\ge2\ (X=A,B,C;\ j=1,2,3).
\end{equation}
\end{proposition}
\begin{proof}
All modes are nonsplit by \cref{lem:H345-splitting}, and $AB,AC$
are saturated. Hence no pure block has CP rank one by
\cref{lem:triple-saturated-pairs}.
For $d=4$, \cref{lem:H345-four-support} already gives
$\kappa_{AB}=\kappa_{AC}=0$.

Suppose $0\ne\matr N\in\set L_{BC}(\tens P)\cap\cK$.
The small-mode support lemma puts $\matr N$ on one block $p$, with
$\rank\matr N\le2$. For $d=5$, this contradicts
\cref{cor:single-block-slice}. It remains to treat $d=4$.
Since both matrix-factor lists are bases, write
$\matr N=\sum n_i\vb_i\otimes\vc_i$, choose $n_1\ne0$, and set
\begin{equation}\label{eq:H345-555-peel}
\tens Q\eqdef\tens P-(\va_1/n_1)\otimes\matr N
=\sum_{i=2}^{5}(\va_i-n_i\va_1/n_1)\otimes\vb_i\otimes\vc_i.
\end{equation}
Write $n_i=\phi(\va_i)$ for the functional $\phi$ inducing this slice
of $\tens P$. The modified factors lie in $\ker\phi$, so
$\dim\Ess_A(\tens Q)=3$ by the affine lower bound.
Its three pure-line contractions are block diagonal.
If the other dimensions are both four, a contraction nonzero on all
four terms has rank four; its row and column spaces show both modes
split. If either dimension is three,
\eqref{eq:triple-two-three-modes} gives the same conclusion.
In the $B$- and $C$-modes, restoring the correction
$(\va_1/n_1)\otimes\matr N$ enlarges the mode spaces only inside
physical block $p$, since $\matr N$ is supported on that block.
Hence the splitting of these two modes is preserved, contrary to
\cref{lem:H345-splitting}.
This proves $\kappa_{BC}=0$; for $d=5$, permuting the modes gives
all three zero intersections.

Finally suppose $d^{\mathrm{proj}}_{A,p}=1$. Its thin contraction is
$\matr M_p\oplus\matr N_q\oplus\matr N_r$.
The pure-block CP rank is at least two, while nonsplitting bounds each
$B,C$ pure intersection by two. Therefore exactly two original indices
are active; after relabelling them as $1,2$, put
$\matr Q_i\eqdef\vb_i\otimes\vc_i$ for $i=1,2$. We then have
\[
\rank\matr M_p=2,\  \matr N_q=\matr N_r=0,\ 
\matr Q_1,\matr Q_2\in\cV_p\otimes\cV_p
\]
and the two matrices are independent and satisfy
\[
\set D\eqdef\Span\{\matr Q_1,\matr Q_2\}\subseteq\cK_p+\FF\matr M_p,
\ \  \dim(\set D\cap\cK_p)\ge1.
\]
For $d=5$, both are slices of $\tens P$, contradicting $\kappa_{BC}=0$.
For $d=4$, put $\set H\eqdef\Span\{\va_3,\va_4,\va_5\}$.
The inactive factors lie in $\cV_q\oplus\cV_r$, so
$2\le\dim\set H\le3$. If $\dim\set H=2$, the first two matrices are
again separate slices of $\tens P$, giving the same contradiction.
Thus $\dim\set H=3$.
Projection onto $q\oplus r$ leaves an affine tensor of CP rank three and
mode type $(3,3,3)$: the other three $B,C$ factors remain independent
after projection. By \cref{lem:H345-R333}, a side pure block has
the same quadratic quotients as a rank-one representative. The
unchanged original block then contradicts
\cref{lem:triple-saturated-pairs}. Hence $d^{\mathrm{proj}}_{A,j}\ge2$ for all $j$.
\end{proof}

\paragraph{Fixed-source reduction.}
Combining
\cref{prop:triple-H0,prop:triple-H1,prop:triple-H2,prop:triple-445,prop:triple-455},
we have $s(\tens V_3)\ge167$. If a minimal $167$-term symmetric
decomposition existed, its original residual would have CP rank five
and, up to mode permutation, mode type $(4,4,5)$, $(4,5,5)$, or
$(5,5,5)$. In every case,
$\kappa_{AB}=\kappa_{AC}=\kappa_{BC}=0$. In either type $(4,4,5)$
or $(4,5,5)$, some four-dimensional mode projects two-dimensionally
onto every block; in type $(5,5,5)$, every mode projects with
dimension at least two onto every block.

\subsection{Common local functionals}
The common-functional criterion now excludes the two four-dimensional cases.

\begin{corollary}[Exclusion of $(4,4,5)$ and $(4,5,5)$]
\label{cor:triple-four}
No residual compatible with a minimal $167$-term symmetric decomposition
has either mode type.
\end{corollary}
\begin{proof}
By \cref{prop:triple-445,prop:triple-455}, choose a four-dimensional
mode whose three projection dimensions are two. Then
\cref{lem:triple-functional} gives a nonzero quadratic slice, contrary
to the vanishing quadratic intersections in
\cref{prop:triple-445,prop:triple-455}.
\end{proof}
\subsection{Local square congruences}
We need two further linear maps on the single-copy base. Write
$\operatorname{pr}_1,\operatorname{pr}_2:\set{G}\to\FF^3$
for the two coordinate projections in \eqref{eq:base-data}.
Define $\matr{S}_1\eqdef\matr{L}$, and define $\matr{S}_2,\matr{S}_3$
by the table
\[
\begin{array}{c|cc}
&\set{G}^{(1)}&\set{G}^{(2)}\\\hline
\matr{S}_1&\operatorname{pr}_2&\operatorname{pr}_1\\
\matr{S}_2&\operatorname{pr}_1&\operatorname{pr}_2\\
\matr{S}_3&-\operatorname{pr}_2&-\operatorname{pr}_2.
\end{array}
\]
These are the three selectors from
\cite[proof of Theorem~4.2]{Lovitz26:arxiv}, denoted there by
$R_1,R_2,R_3$.
All restrictions agree on the common line $\FF\vc$.
Each map sends the square generators to coordinate squares, and
the argument for \eqref{eq:six} applies to each map. We call
$\va^{(1)},\va^{(2)},\vc$ the three standard directions; their values
under the selectors are
\begin{equation}\label{eq:triple-signs}
\begin{array}{c|ccc}
&\va^{(1)}&\va^{(2)}&\vc\\\hline
\matr{S}_1&-\vect{\alpha}&\vect{\alpha}&\vect{\alpha}\\
\matr{S}_2&\vect{\alpha}&-\vect{\alpha}&\vect{\alpha}\\
\matr{S}_3&\vect{\alpha}&\vect{\alpha}&-\vect{\alpha}.
\end{array}
\end{equation}
In particular, all three coordinates of each $\matr{S}_q\vd$
are nonzero when $\vd$ is a factor of a rank-one tensor in $\cA_0$.
Let $\cK^{(i)}$ be the nine-dimensional square space in internal
sector $\set{G}^{(i)}$. We use the following known sector
identities and their direct two-sector consequence:
\begin{equation}\label{eq:triple-sector}
\cK^{(i)}\cap\vc\set{G}^{(i)}=0,\ \ 
\cK_0\cap\vc\cV_0=0,\ \ 
(\va^{(1)})^2\equiv(\va^{(2)})^2\equiv\vc^2\pmod{\cK_0}.
\end{equation}
The first equality is \cite[Lemma~4.3]{Lovitz26:arxiv}.
It implies that $\cK^{(i)}\to\Sym^2((\set{G}')^{(i)})$ is injective;
projecting to the two internal sectors then gives the second equality.
The congruences are the identities
$\sum_{p,q}\matr{M}_{i,p,q}=(\va^{(i)})^2-\vc^2$
from \cite[Lemma~4.1]{Lovitz26:arxiv}.
The real statements follow by restriction from the complex ones.

\begin{lemma}[Rigidity of square congruences]\label{lem:triple-fiber}
Suppose $\vu\otimes\vd^2\in\cA_0$. If the vectors $\vx$ satisfying
\begin{equation}\label{eq:triple-fiber}
\vx^2-\lambda\vd^2\in\cK_0
\ \text{for some }\lambda\in\FF\setminus\{0\}
\end{equation}
span a space of dimension at least three, then $\vd$ is parallel
to one of $\va^{(1)},\va^{(2)},\vc$. For any of these three
directions, the solution set consists precisely of the nonzero vectors
on their three lines.
\end{lemma}
\begin{proof}
Write $\cV_0=\FF\vc\oplus(\set{G}')^{(1)}\oplus(\set{G}')^{(2)}$.
Every matrix in $\cK_0$ has zero
$(\set{G}')^{(1)}\otimes(\set{G}')^{(2)}$ component.
First suppose $\vd\in\set{G}^{(1)}$; the other sector is identical.
The cross-sector component of \eqref{eq:triple-fiber} forces
$\vx$ to lie in one sector. If both lie in the first sector,
write $\vd=(\vv,\vw)$ and $\vx=(\vv',\vw')$ in six coordinates.
Both coordinate triples of $\vd$ are nonzero in every entry,
by the selector property. The off-diagonal entries of the two
diagonal blocks give
\[
\vv'=t\vv,\ \ \vw'=s\vw,\ \  t^2=s^2=\lambda.
\]
There are therefore at most two one-dimensional solution subspaces
in that sector.
If a direction also occurs in the other sector, projection onto
the sector quotients gives
$\vd^2\in\cK^{(1)}+\FF\vc^2$.
By \cite[Lemma~4.3]{Lovitz26:arxiv}, the vector $\vd$ is parallel to
$\va^{(1)}$, to $\vc$, or to one of the vectors $\vw_{pq}$.
The last possibility is excluded by $\vd^2\notin\cK_0$.

Now suppose $\vd=\tau\vc+\vv+\vw$, where
$0\ne\vv\in(\set{G}')^{(1)}$ and
$0\ne\vw\in(\set{G}')^{(2)}$.
We may normalize \eqref{eq:triple-fiber} to $\lambda=1$.
Over $\CC$ this uses a square root. Over $\RR$, for the three nonzero
coordinate ratios $r_1,r_2,r_3$, the three off-diagonal selector
equations give
\[
r_1r_2=r_1r_3=r_2r_3=\lambda.
\]
Hence $r_1=r_2=r_3$ and $\lambda=r_i^2>0$, so the normalization
is real.
The cross-sector component now gives
$\vx=\tau'\vc+t\vv+t^{-1}\vw,\ \  t\ne0.$ 
If $t^2=1$, \eqref{eq:triple-sector} gives $\vx=\pm\vd$.
If a solution with $t^2\ne1$ exists, injectivity of the two sector
projections uniquely determines $\vb_i\in(\set G')^{(i)}$ and
$\eta_i\in\FF$ for which
\[
\matr{Q}_1^*\eqdef\vv^2+2\vc\vb_1+\eta_1\vc^2\in\cK^{(1)},
\ \ 
\matr{Q}_2^*\eqdef\vw^2+2\vc\vb_2+\eta_2\vc^2\in\cK^{(2)}.
\]
Comparing the $\vc(\set{G}')^{(i)}$ components gives
$\vb_1=\beta_1\vv$ and $\vb_2=\beta_2\vw$, together with
\[
\tau't-\tau=(t^2-1)\beta_1,\ \ 
\tau'/t-\tau=(t^{-2}-1)\beta_2.
\]
Thus $\tau=\beta_1+\beta_2$ and
$\tau'=t\beta_1+t^{-1}\beta_2$.
Uniqueness of $\matr{Q}_i^*$ makes the $\beta_i$ independent
of the chosen solution. Every solution lies in
$\Span\{\vv+\beta_1\vc,\ \vw+\beta_2\vc\},$ 
contrary to the three-dimensional span hypothesis.
If no such solution exists, all solutions lie on one line, giving
the same contradiction. This proves the first assertion.
For a standard direction, the cross-sector component puts each solution
in one sector. Apply \cite[Lemma~4.3]{Lovitz26:arxiv} there;
the lines spanned by the $\vw_{pq}$ have zero class and are excluded
since $\lambda\ne0$.
The remaining possibilities are precisely the three stated lines. Their occurrence
follows from \eqref{eq:triple-sector}.
\end{proof}

\subsubsection{Fixed-source normalization}
Using the quotient projection $\overline{\matr p}_j$ from
\cref{lem:triple-good}, the complete coefficient identity gives
\begin{equation}\label{eq:triple-quotient-complete}
(\id\otimes\overline{\matr p}_j)\tens G_{BC}
=(\id\otimes\overline{\matr p}_j)\tens P.
\end{equation}
Put $\va_{0,j}\eqdef\va_j^{(1)}+\va_j^{(2)}+\vc_j$.

\begin{lemma}[Local normalization]
\label{lem:triple-normalization}
\begin{enumerate}
\item Suppose an original-source quotient is
$\vu\otimes[\vd^2]_j$, with $\vu\otimes\vd^2\in\cA_j$, and
every original local matrix belongs to $\cK_j+\FF\vd^2$.
After rescaling, its quotient is $\va_{0,j}\otimes[\vd_j^2]_j$,
where $\vd_j$ is one of the three standard vectors.
The nonzero local quotient matrices are scalar multiples of the three
standard squares, and all three directions occur.
\item For an original five-term residual of a minimal $167$-term symmetric
decomposition with $d_A=d_B=d_C=5$, we have
\begin{equation}\label{eq:triple-normalization}
\dim\set U_{A,j}=1\ \Longrightarrow\
\set U_{A,j}=\FF\va_{0,j},\ \
(\id\otimes\overline{\matr p}_j)\tens P
=\va_{0,j}\otimes[\vd_j^2]_j,
\end{equation}
where $\vd_j$ may be chosen in
$\matr p_j\set U_B\cap\matr p_j\set U_C$.
For every standard direction, we have
\begin{equation}\label{eq:triple-standard-plane}
\cK_j\cap\Sym^2\Span\{\va_{0,j},\vd_j\}=0.
\end{equation}
\end{enumerate}
\end{lemma}
\begin{proof}
(i) Every original local matrix lies in
$\cK_j+\FF\vd^2\subseteq\Sym^2\cV_j$.
A nonzero rank-one symmetric matrix is a nonzero scalar times a square.
Thus, for each matrix with nonzero class, write
\[
\matr Q_i=\mu_i\vx_i^2,\ \ 
[\matr Q_i]_j=\lambda_i[\vd^2]_j,\ \
\mu_i\lambda_i\ne0.
\]
Real scalars $\mu_i$ need not be positive. Put
$\set H\eqdef\Span\{\vx_i\}$. Complete coefficient comparison gives
$\vu=\sum_i\lambda_i\vh_i$; applying the selected retained-mode map
and then $\matr p_j$ shows that $\vu\in\set H$. If $\dim\set H=1$,
replacing $\vd^2$ by a normalized $\matr Q_i$
gives a cube in $\cA_j$. If $\dim\set H=2$, choose independent
$\vx_i,\vx_\ell$. Neither can be parallel to $\vu$, for the same
reason. Thus the representative
$\vu\otimes\matr Q_i/\lambda_i$ has factor plane $\set H$, while
\[
0\ne\matr Q_i/\lambda_i-\matr Q_\ell/\lambda_\ell
\in\cK_j\cap\Sym^2\set H
\]
has rank two, contrary to \cref{lem:local}. Hence $\dim\set H\ge3$. 
By \cref{lem:triple-fiber}, the $\vx_i$ lie on the three standard
lines. They span $\set H$, so all three occur. Normalize $\vd$ to a
standard vector and write
$\vu=x_1\va_j^{(1)}+x_2\va_j^{(2)}+x_3\vc_j$.
The selector equations become
\[
\begin{pmatrix}-1&1&1\\1&-1&1\\1&1&-1\end{pmatrix}
\begin{pmatrix}x_1\\x_2\\x_3\end{pmatrix}
=\begin{pmatrix}1\\1\\1\end{pmatrix}.
\]
The determinant is $4$, whence $x_1=x_2=x_3=1$.

(ii) By the block conditions, the tensor
$\tens\Psi\eqdef(\id\otimes\overline{\matr p}_j)\tens P$ belongs to
$\set U_{A,j}\otimes\im\overline{\matr p}_j$ and is nonzero. Write
$\tens\Psi=\vu\otimes\overline{\matr N}$. The five original $A$ factors form a basis;
taking dual coefficients gives
$[\vb_i^{[j]}\otimes\vc_i^{[j]}]_j=\lambda_i\overline{\matr N}$.
For $\lambda_i\ne0$, we have
$(\vu/\lambda_i)\otimes\vb_i^{[j]}\otimes\vc_i^{[j]}\in\cA_j$.
Saturation and \cref{lem:triple-good} make this matrix a square.
All original local classes lie on the same line, so (i) applies.
Choosing its standard direction from this residual matrix gives the
asserted membership in both projected mode spaces.
Finally $\va_{0,j}\otimes\vd_j^2\in\cA_j$ by
\cref{eq:triple-sector,eq:single-representatives};
\cref{lem:local} gives \eqref{eq:triple-standard-plane}.
\end{proof}

\subsection{The final five-dimensional obstruction}
\begin{proposition}[Exclusion of $(5,5,5)$]\label{prop:triple-five}
No residual compatible with a minimal $167$-term symmetric decomposition
has multilinear rank $(5,5,5)$.
\end{proposition}
\begin{proof}
By \cref{prop:triple-455}, all three quadratic intersections vanish
and every block projection in every mode has dimension at least two.
We first show that each mode has at least two blocks $j$ with
$\dim(\set U_X\cap\cV_j)=1$ and $\dim\matr p_j\set U_X=2$.
Fix a mode $\set{U}$ and write
$t_j\eqdef\dim(\set{U}\cap\cV_j)$ and
$d_j^{\mathrm{proj}}\eqdef\dim\matr{p}_j\set{U}$.
Let $\set{S}\eqdef\bigoplus_j(\set{U}\cap\cV_j)$ and choose
$\set{U}=\set{S}\oplus\set{H}$. Subtracting the pure-block
components from the chosen complement makes
$\matr{p}_j\set{H}$ disjoint from $\set{U}\cap\cV_j$.
Put $\Lambda_j(\set U)\eqdef
\im\bigl((\matr p_j|_{\set U})^*\bigr)\subseteq\set U^*$.
In the dual decomposition, choose $\set{D}_j\subseteq\set{H}^*$ so that
\[
\Lambda_j(\set U)=(\set{U}\cap\cV_j)^*\oplus\set{D}_j,
\ \ 
\dim\set{D}_j=d_j^{\mathrm{proj}}-t_j.
\]
By \cref{lem:triple-functional} and the zero quadratic
intersection, we have $\set{D}_1\cap\set{D}_2\cap\set{D}_3=0$.
If $\dim\set{S}=5$, the mode splits completely; then
$d_j^{\mathrm{proj}}=t_j\ge2$ for all three blocks, which is impossible.
Thus the pure-intersection dimensions are $(1,1,1)$ or a permutation
of $(2,1,1)$.

In type $(1,1,1)$, we have $\dim\set{H}=2$, and every $\set{D}_j$ is
nonzero. At most one can equal $\set{H}^*$, since two such
spaces would have nonzero intersection with the third.
At least two have dimension one, giving $t_j=1$ and $d_j^{\mathrm{proj}}=2$.
In type $(2,1,1)$, we have $\dim\set{H}=1$, and the two blocks with
$t_j=1$ have $d_j^{\mathrm{proj}}=2$, because their projections are thick.
Two subsets of $\{1,2,3\}$, each of size at least two, intersect.

Choose such a common block $j$ for modes $A,B$, and apply
\cref{lem:triple-normalization}
to both modes, retaining the same original maps. Then
\[
\set{U}_A\cap\cV_j=\set{U}_B\cap\cV_j=\FF\va_{0,j},
\ \ 
\matr{p}_j\set{U}_B=\Span\{\va_{0,j},\vd_j\},
\]
where $\va_{0,j},\vd_j$ are independent and
$\vd_j\in\matr{p}_j\set{U}_C$.
Write $\set{B}_j\eqdef\matr{p}_j\set{U}_B$ and
$\set{C}_j\eqdef\matr{p}_j\set{U}_C$.
If $\matr N\in\cK_j\cap(\set{B}_j\otimes\set{C}_j)$, then $\matr N$
is symmetric, so
$\operatorname{row}\matr N=\operatorname{col}\matr N
\subseteq\set{B}_j\cap\set{C}_j$; hence
$\matr N\in\Sym^2(\set{B}_j)$. Since
$\set{B}_j=\Span\{\va_{0,j},\vd_j\}$,
\eqref{eq:triple-standard-plane} gives $\matr N=0$. Thus
$\cK_j\cap(\set{B}_j\otimes\set{C}_j)=0.$
The quotient identity in \eqref{eq:triple-normalization} and the
mode spaces consequently imply the exact equality
\begin{equation}\label{eq:triple-pure-exact}
\tens{P}_{jjj}=\va_{0,j}\otimes\vd_j\otimes\vd_j.
\end{equation}
Indeed, the difference lies both in $\cV_j\otimes\cK_j$ and
in $\cV_j\otimes\set{B}_j\otimes\set{C}_j$.

All original pairs are saturated by \cref{lem:triple-saturation}.
The rank-one pure block \eqref{eq:triple-pure-exact} therefore
contradicts \cref{lem:triple-saturated-pairs}.
\end{proof}

\begin{proof}[Proof of \cref{thm:main} for $k=3$]
The CP rank, upper bound, and concision are given by
\eqref{eq:triple-initial}. By \cref{prop:triple-H0}, a smaller
symmetric rank would have to equal $167$. For such a decomposition,
\cref{prop:triple-H1,prop:triple-H2} leave only the mode types
$(4,4,5)$, $(4,5,5)$, and $(5,5,5)$.
\Cref{cor:triple-four} excludes the first two, and
\cref{prop:triple-five} excludes the last.
Therefore $s(\tens{V}_3)=168$ and
$\DeltaT_{\FF}(\tens{V}_3)=168-165=3$.
Rationality follows directly from the construction of
$\tens{T}_0$ and the use of independent copies.
\end{proof}

\section{Conclusion}
\label{sec:conclusion}

By combining the direct-sum strategy proposed in
\cite[Section~6]{shitov2018counterexample} with Lovitz's rational
example \cite{Lovitz26:arxiv}, we have constructed third-order
symmetric tensors with Comon rank gaps two and three over both
$\RR$ and $\CC$. 
This naturally leads to
the following question.

\begin{problem}
Can one construct over $\RR$ or $\CC$ a family of third-order symmetric tensors with unbounded Comon rank gaps, either by direct sums or by another method?
\end{problem}

\noindent\textbf{Acknowledgments.} 
The author acknowledges the assistance of ChatGPT 5.6 Sol throughout the mathematical development and preparation of this manuscript. 
The author takes full responsibility for the content of this paper. 
This work was supported by the Guangdong Basic and Applied Basic
Research Foundation (2026A1515050001, 2026A1515010188).

\bibliographystyle{siamplain}
\bibliography{ReferenceTensor}

\begin{thebibliography}{10}

\bibitem{alexeev2011tensor}
{\sc B.~Alexeev, M.~A. Forbes, and J.~Tsimerman}, {\em Tensor rank: Some lower
  and upper bounds}, in 2011 IEEE 26th Annual Conference on Computational
  Complexity, IEEE, 2011, pp.~283--291.

\bibitem{ballico2011lowrank}
{\sc E.~Ballico}, {\em Tensor ranks and symmetric tensor ranks are the same for
  points with low symmetric tensor rank}, Archiv der Mathematik, 96 (2011),
  pp.~531--534.

\bibitem{ballico2013tensor}
{\sc E.~Ballico and A.~Bernardi}, {\em Tensor ranks on tangent developable of
  {Segre} varieties}, Linear and Multilinear Algebra, 61 (2013), pp.~881--894.

\bibitem{buczynski2020additivity}
{\sc J.~Buczy{\'n}ski, E.~Postinghel, and F.~Rupniewski}, {\em On {Strassen's}
  rank additivity for small three-way tensors}, SIAM Journal on Matrix Analysis
  and Applications, 41 (2020), pp.~106--133.

\bibitem{carlini2015symmetricstrassen}
{\sc E.~Carlini, M.~V. Catalisano, and L.~Chiantini}, {\em Progress on the
  symmetric {Strassen} conjecture}, Journal of Pure and Applied Algebra, 219
  (2015), pp.~3149--3157.

\bibitem{carlini2018ecomputability}
{\sc E.~Carlini, M.~V. Catalisano, L.~Chiantini, A.~V. Geramita, and Y.~Woo},
  {\em Symmetric tensors: rank, {Strassen's} conjecture and e-computability},
  arXiv preprint arXiv:1506.03176,  (2015).

\bibitem{casarotti2020comon}
{\sc A.~Casarotti, A.~Massarenti, and M.~Mella}, {\em On {Comon’s and
  Strassen’s} conjectures}, Mathematics, 6 (2018), p.~217.

\bibitem{comon2004canonical}
{\sc P.~Comon}, {\em Canonical tensor decompositions}, in Workshop on Tensor
  Decompositions, Palo Alto, CA, 2004.

\bibitem{comon2014tensors}
{\sc P.~Comon}, {\em Tensors: a brief introduction}, IEEE Signal Processing
  Magazine, 31 (2014), pp.~44--53.

\bibitem{comon2008symmetric}
{\sc P.~Comon, G.~Golub, L.-H. Lim, and B.~Mourrain}, {\em Symmetric tensors
  and symmetric tensor rank}, SIAM Journal on Matrix Analysis and Applications,
  30 (2008), pp.~1254--1279.

\bibitem{draisma2024erratum}
{\sc J.~Draisma}, {\em Erratum: A counterexample to {Comon’s} conjecture},
  SIAM Journal on Applied Algebra and Geometry, 8 (2024), pp.~225--225.

\bibitem{friedland2016remarks}
{\sc S.~Friedland}, {\em Remarks on the symmetric rank of symmetric tensors},
  SIAM Journal on Matrix Analysis and Applications, 37 (2016), pp.~320--337.

\bibitem{hopcroft1971minimizing}
{\sc J.~E. Hopcroft and L.~R. Kerr}, {\em On minimizing the number of
  multiplications necessary for matrix multiplication}, SIAM Journal on Applied
  Mathematics, 20 (1971), pp.~30--36.

\bibitem{kolda2009tensor}
{\sc T.~G. Kolda and B.~W. Bader}, {\em Tensor decompositions and
  applications}, SIAM Review, 51 (2009), pp.~455--500.

\bibitem{landsberg2012tensors}
{\sc J.~M. Landsberg}, {\em Tensors: geometry and applications}, vol.~128,
  American Mathematical Soc., 2012.

\bibitem{Lovitz26:arxiv}
{\sc B.~Lovitz}, {\em A 27×27×27 counterexample to {Comon’s} conjecture},
  tech. report, 2026.
\newblock {a}rXiv:2609.28292v1.

\bibitem{rhodes2010concise}
{\sc J.~A. Rhodes}, {\em A concise proof of kruskal’s theorem on tensor
  decomposition}, Linear Algebra and its Applications, 432 (2010),
  pp.~1818--1824.

\bibitem{rupniewski2024additivity}
{\sc F.~Rupniewski}, {\em {Strassen's} rank additivity for small tensors,
  including tensors of rank less or equal 7}, Linear Algebra and its
  Applications, 698 (2024), pp.~448--478.

\bibitem{seigal2020cubic}
{\sc A.~Seigal}, {\em Ranks and symmetric ranks of cubic surfaces}, Journal of
  Symbolic Computation, 101 (2020), pp.~304--317.

\bibitem{shitov2018counterexample}
{\sc Y.~Shitov}, {\em A counterexample to {Comon's} conjecture}, SIAM Journal
  on Applied Algebra and Geometry, 2 (2018), pp.~428--443.

\bibitem{shitov2019directsum}
{\sc Y.~Shitov}, {\em Counterexamples to {Strassen's} direct sum conjecture},
  Acta Mathematica,  (2019).

\bibitem{shitov2020comon}
{\sc Y.~Shitov}, {\em {Comon's conjecture} over the reals}, viXra:2009.0134,
  (2020).

\bibitem{shitovhigher}
{\sc Y.~Shitov}, {\em Higher rank substitutions for tensor decompositions {II.}
  {Comon's} conjecture},  (2024).

\bibitem{wang2023lower}
{\sc K.~Wang and A.~Seigal}, {\em Lower bounds on the rank and symmetric rank
  of real tensors}, Journal of Symbolic Computation, 118 (2023), pp.~69--92.

\bibitem{woo2014cases}
{\sc Y.~Woo}, {\em Some cases on {Strassen} additive conjecture}, arXiv
  preprint arXiv:1406.2213,  (2014).

\bibitem{zhang2016comon}
{\sc X.~Zhang, Z.-H. Huang, and L.~Qi}, {\em {Comon's conjecture}, rank
  decomposition, and symmetric rank decomposition of symmetric tensors}, SIAM
  Journal on Matrix Analysis and Applications, 37 (2016), pp.~1719--1728.

\end{thebibliography}

\end{document}